\documentclass[11pt]{amsart}

\usepackage{hyperref}
\hypersetup{nesting=true,debug=true,naturalnames=true}
\usepackage{graphicx,amssymb,upref}
\usepackage[usenames,dvipsnames]{pstricks}
\usepackage{epsfig}
\usepackage{color}
\usepackage[utf 8]{inputenc}
\usepackage[T1]{fontenc}
\usepackage{amsmath, amsthm}
\usepackage[english]{babel}
\usepackage{amssymb}
\usepackage{latexsym}
\usepackage{graphicx,amssymb,upref,mathtools}
\usepackage[all]{xy}
\usepackage{tikz}
\usepackage{yhmath}
\usetikzlibrary{matrix} 

\date{\today}

\newtheorem{theorem}{Theorem}[section]
\newtheorem{lemma}[theorem]{Lemma}

\newtheorem{corollary}[theorem]{Corollary}
\theoremstyle{definition}
\newtheorem{definition}[theorem]{Definition}

\newtheorem{remark}[theorem]{Remark}

\newcommand{\ot}{\otimes}
\newcommand{\co}{\circ}

\DeclareMathOperator{\hatast}{\,\hat{\ast}\,}

\DeclareMathOperator{\Hom}{Hom}

\title{Categorical isomorphisms for Hopf braces} 

\begin{document}
	
\maketitle
	
\begin{center}
	{\bf Jos\'e Manuel Fern\'andez Vilaboa$^{1}$, Ram\'on
	Gonz\'{a}lez Rodr\'{\i}guez$^{2}$ and Brais Ramos P\'erez$^{3}$}.
	\end{center}
	
	\vspace{0.4cm}
	\begin{center}
	{\small $^{1}$ [https://orcid.org/0000-0002-5995-7961].}
	\end{center}
	\begin{center}
	{\small  CITMAga, 15782 Santiago de Compostela, Spain.}
	\end{center}
	\begin{center}
	{\small  Universidade de Santiago de Compostela. Departamento de Matem\'aticas,  Facultade de Matem\'aticas, E-15771 Santiago de Compostela, Spain. 
	\\ email: josemanuel.fernandez@usc.es.}
	\end{center}
	\vspace{0.2cm}
	
	\begin{center}
	{\small $^{2}$ [https://orcid.org/0000-0003-3061-6685].}
	\end{center}
	\begin{center}
	{\small  CITMAga, 15782 Santiago de Compostela, Spain.}
	\end{center}
	\begin{center}
	{\small  Universidade de Vigo, Departamento de Matem\'{a}tica Aplicada II,  E. E. Telecomunicaci\'on,
	E-36310  Vigo, Spain.
	\\email: rgon@dma.uvigo.es}
	\end{center}

\vspace{0.2cm}

   \begin{center}
   	{\small $^{3}$ [https://orcid.org/0009-0006-3912-4483].}
   \end{center}
	\begin{center}
	{\small  CITMAga, 15782 Santiago de Compostela, Spain. \\}
	\end{center}
    \begin{center}
	{\small  Universidade de Santiago de Compostela. Departamento de Matem\'aticas,  Facultade de Matem\'aticas, E-15771 Santiago de Compostela, Spain. 
	\\email: braisramos.perez@usc.es}
	\end{center}
	\vspace{0.2cm}

\begin{abstract}  
In this paper, we introduce the category of brace triples in a braided monoidal setting and prove that it is isomorphic to the category of \texttt{s}-Hopf braces, which are a generalization of cocommutative Hopf braces. After that, we obtain a categorical isomorphism between finite cocommutative Hopf braces and a certain subcategory of cocommutative post-Hopf algebras, which suppose an expansion to the braided monoidal setting of the equivalence obtained by Y. Li, Y. Sheng and R. Tang in \cite{LST} for the category of vector spaces over a field $\mathbb{K}$.
\end{abstract} 

\vspace{0.2cm}



\section*{Introduction}

Hopf braces were born in \cite{AGV} as the quantum version of skew braces, introduced by L. Guarnieri and L. Vendramin in \cite{GV}. The importance of these objects is fundamentally due to the fact that they provide solutions of the Quantum Yang-Baxter equation, that is a relevant subject in mathematical physics (see \cite{Yang} and \cite{Baxter}). Despite the simplicity of its formulation, finding solutions of the Yang-Baxter equation is not an easy task. In fact, the problem of classifying all the solutions of the equation is still open and different approaches have been proposed since the end of the last century. One of them was proposed by Drinfel'd in \cite{Dr1}, that consists on studying non-degenerate set-theoretical solutions. Research into this kind of solutions with the involutive property was what gave rise to the concept of brace introduced by Rump in \cite{Rump} for which skew braces are a generalization. So, a skew brace consists on two different group structures, $(G,.)$ and $(G,\star)$, satisfying the following compatibility condition 
\begin{equation}\label{ccskewbraces}g\star(h. t)=(g\star h). g^{-1}. (g\star t)\end{equation}
for all $g,h,t\in G$, where $g^{-1}$ denotes the inverse of $g$ with respect to the group structure $(G,.)$. These structures are useful to find non-degenerate solutions of the Yang-Baxter equation not neccesarily involutive. The linearization of skew braces give rise to the notion of Hopf brace defined by I. Angiono, C. Galindo and L. Vendramin in \cite{AGV}: If $(H,\epsilon,\Delta)$ is a coalgebra, a Hopf brace structure over $H$ consists on two different Hopf algebra structures,
$$
H_{1}=(H,1,\cdot,\epsilon,\Delta,\lambda),\;\;\; H_{2}=(H,1_{\circ},\circ,\epsilon,\Delta,S),
$$
where $\lambda$ and $S$ denote the antipodes, satisfying the following compatibility condition 
\[g\circ(h\cdot t)=(g_{1}\circ h)\cdot\lambda(g_{2})\cdot(g_{3}\circ t)\quad\forall g,h,t\in H\]
which generalizes \eqref{ccskewbraces}. Moreover, as was pointed in \cite[Corollary 2.4]{AGV}, cocommutative Hopf braces give rise to solutions of Yang-Baxter equation too.

On the other hand, without going into detail, it is not irrelevant to highlight the relationship between Hopf braces and invertible 1-cocycles, which are nothing more than coalgebra isomorphisms between two different Hopf algebras, $\pi\colon H\rightarrow B$, such that $B$ is a $H$-module algebra. In \cite[Theorem 1.12]{AGV} it is proved that the category of Hopf braces with $H_{1}$ fixed is equivalent to the category of invertible 1-cocycles $\pi\colon H_{1}\rightarrow B$. Moreover, this result remains valid in the case that $H_{1}$ is not fixed (see \cite{VRBAMod}).

Therefore, motivated by the fact that cocommutative Hopf braces induce solutions of Quantum Yang-Baxter equation, in this paper we study another objects that characterise the structure of Hopf braces in the cocommutative setting. So, given $\sf{C}$ a braided monoidal category, in Section \ref{sect2} we introduce the category of brace triples (see Definition \ref{BTdef}) and the category of \texttt{s}-Hopf braces (see Definition \ref{sHbr}). Note that \texttt{s}-Hopf braces generalize cocommutative Hopf braces because both categories are the same under cocommutativity assumption (see Remark \ref{sHB-cocHB}). After that, a functor from brace triples to \texttt{s}-Hopf braces is constructed explicitly (see Theorem \ref{Prop Brace TB}), and another from \texttt{s}-Hopf braces to brace triples (see Theorem \ref{sHB to BT}), ending the section with the main Theorem \ref{iso1} where we prove that the previous correspondence give rise to a categorical isomorphism. As a consequence (see Corollary \ref{Cor iso cocHB cocTB}), we obtain a categorical isomorphism between cocommutative Hopf braces and cocommutative brace triples, that is to say, a cocommutative Hopf brace is no more than a cocommutative Hopf algebra $H$ together with a pair of morphisms, $T_{H}\colon H\rightarrow H$ and $\gamma_{H}\colon H\otimes H\rightarrow H$, satisfying some compatibility conditions between them and the Hopf algebra structure over $H$. 

In Section \ref{sec3}, we introduce the notion of post-Hopf algebra in a braided monoidal category $\sf{C}$ (see Definition \ref{PHopf algebra}), which generalizes the one introduced by Y. Li, Y. Sheng and R. Tang in \cite{LST} for a category of vector spaces over a field $\mathbb{K}$ (see also \cite{BGST}). After proving some interesting properties of these objects that can be deducted from the definition, we obtain a functor from finite brace triples to post-Hopf algebras (see Theorem \ref{inversebetaBT}). At this point, hypothesis of cocommutativity acquires significant importance and it is essential, together with technical condition \eqref{property for betatilde2}, to prove the existence of a functor from the category of cocommutative post-Hopf algebras that satisfy \eqref{property for betatilde2} to the category of finite cocommutative Hopf braces (see Theorem \ref{Hhat brace}). Therefore, Theorem \ref{iso2} is the main result of this section, where we prove that this correspondence induces a categorical isomorphism between $\sf{cocPost}\textnormal{-}\sf{Hopf}^{\star}$, the category of cocommutative post-Hopf algebras satisfying \eqref{property for betatilde2}, and finite cocommutative brace triples. So, as a consequence, $\sf{cocPost}\textnormal{-}\sf{Hopf}^{\star}$, finite cocommutative brace triples and finite cocommutative Hopf braces are isomorphic, which suppose a generalization of \cite[Theorem 2.13]{LST} to the braided monoidal setting.

In the following diagram it is posible to consult a summary of the categorical relationships that can be seen along this paper. Detailed notation information will be introduced throughout the paper.
\vspace{-0.5cm}
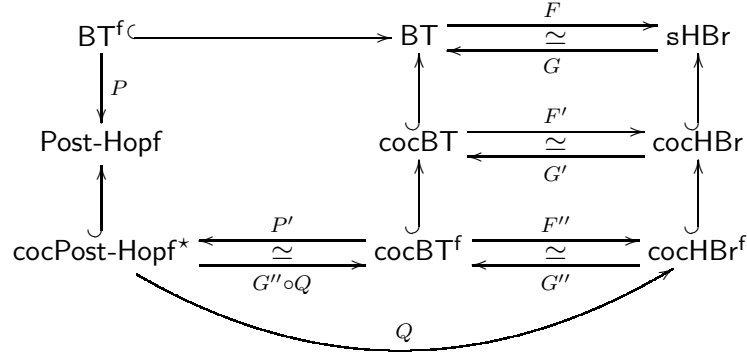
\begin{figure}[h]
\[\xymatrix{
 &\sf{BT^{f}}\ar[d]^-{P}\ar@{^(->}[rr] & &\sf{BT}\ar@<1 ex>[rr]^-{F} &\simeq &\texttt{s}\sf{HBr}\ar@<1 ex>[ll]^-{G} \\
 &\sf{Post}\textnormal{-}\sf{Hopf} & &\sf{cocBT}\ar@{^(->}[u]\ar@<1ex>[rr]^-{F'} &\simeq &\sf{cocHBr}\ar@{^(->}[u]\ar@<1ex>[ll]^-{G'}\\
 &\sf{cocPost}\textnormal{-}\sf{Hopf}^{\star}\ar@/_{13mm}/[rrrr]^-{Q}\ar@<-1ex>[rr]_-{G''\circ Q}\ar@{^(->}[u] &\simeq &\sf{cocBT^{f}}\ar@<-1ex>[ll]_-{P'}\ar@{^(->}[u]\ar@<1ex>[rr]^-{F''} &\simeq &\sf{cocHBr^{f}}\ar@{^(->}[u]\ar@<1ex>[ll]^-{G''}
}
\]
\label{Figure 1}
\caption{Categorical relationships between $\sf{HBr}$, $\sf{BT}$ and $\sf{Post}$-${\sf Hopf}$.}
\end{figure}

\section{Preliminaries}

Throughout this paper we are going to denote by ${\sf  C}$ a strict braided monoidal category with tensor product $\ot$, unit object $K$ and braiding $c$. 

As can be found in \cite{Mac}, a monoidal category is a category ${\sf  C}$ together with a functor $\ot :{\sf  C}\times {\sf  C}\rightarrow {\sf  C}$, called tensor product, an object $K$ of ${\sf C}$, called the unit object, and  families of natural isomorphisms 
$$a_{M,N,P}:(M\ot N)\ot P\rightarrow M\ot (N\ot P),\;\;\;r_{M}:M\ot K\rightarrow M, \;\;\; l_{M}:K\ot M\rightarrow M,$$
in ${\sf  C}$, called  associativity, right unit and left unit constraints, respectively, which satisfy the Pentagon Axiom and the Triangle Axiom, i.e.,
$$a_{M,N, P\ot Q}\co a_{M\ot N,P,Q}= (id_{M}\ot a_{N,P,Q})\co a_{M,N\ot P,Q}\co (a_{M,N,P}\ot id_{Q}),$$
$$(id_{M}\ot l_{N})\co a_{M,K,N}=r_{M}\ot id_{N},$$
where for each object $X$ in ${\sf  C}$, $id_{X}$ denotes the identity morphism of $X$. A monoidal category is called strict if the previous constraints are identities. It is an important result (see for example \cite{K}) that every non-strict monoidal category is monoidal equivalent to a strict one, so the strict character can be assumed without loss of generality. Then, results proved in a strict setting hold for every non-strict monoidal category that include, between others, the category ${\mathbb K}$-{\sf Vect} of vector spaces over a field ${\mathbb K}$, the category $R$-{\sf Mod} of left modules over a commutative ring $R$ or the category of sets, $\sf{Set}$.
For
simplicity of notation, given objects $M$, $N$, $P$ in ${\sf  C}$ and a morphism $f:M\rightarrow N$, in most cases we will write $P\ot f$ for $id_{P}\ot f$ and $f \ot P$ for $f\ot id_{P}$.

A braiding for a strict monoidal category ${\sf  C}$ is a natural family of isomorphisms 
$$c_{M,N}:M\ot N\rightarrow N\ot M$$ subject to the conditions 
$$
c_{M,N\ot P}= (N\ot c_{M,P})\co (c_{M,N}\ot P),\;\;
c_{M\ot N, P}= (c_{M,P}\ot N)\co (M\ot c_{N,P}).
$$

A strict braided monoidal category ${\sf  C}$ is a strict monoidal category with a braiding. These categories were introduced by Joyal and Street in \cite{JS1} (see also  \cite{JS2}) motivated by the theory of braids and links in topology. Note that, as a consequence of the definition, the equalities $c_{M,K}=c_{K,M}=id_{M}$ hold, for all object  $M$ of ${\sf  C}$. Moreover, if $\sf{C}$ is braided with braiding $c$, then $\sf{C}$ is also braided with braiding $c^{-1}$. We will denote by $\overline{\sf{C}}$ the category $\sf{C}$ with braiding $c^{-1}$.

If the braiding satisfies that  $c_{N,M}\co c_{M,N}=id_{M\ot N},$ for all $M$, $N$ in ${\sf  C}$, we will say that ${\sf C}$  is symmetric. In this case, we call the braiding $c$ a symmetry for the category ${\sf  C}$.

In the following definitions we sum up some basic notions in the braided monoidal setting.

\begin{definition}
{\rm 
An algebra in ${\sf  C}$ is a triple $A=(A, \eta_{A},
\mu_{A})$ where $A$ is an object in ${\sf  C}$ and
 \mbox{$\eta_{A}:K\rightarrow A$} (unit), $\mu_{A}:A\otimes A
\rightarrow A$ (product) are morphisms in ${\sf  C}$ such that
$\mu_{A}\circ (A\otimes \eta_{A})=id_{A}=\mu_{A}\circ
(\eta_{A}\otimes A)$, $\mu_{A}\circ (A\otimes \mu_{A})=\mu_{A}\circ
(\mu_{A}\otimes A)$. Given two algebras $A= (A, \eta_{A}, \mu_{A})$
and $B=(B, \eta_{B}, \mu_{B})$, a morphism  $f:A\rightarrow B$ in {\sf  C} is an algebra morphism if $\mu_{B}\circ (f\otimes f)=f\circ \mu_{A}$, $ f\circ
\eta_{A}= \eta_{B}$. 

If  $A$, $B$ are algebras in ${\sf  C}$, the tensor product
$A\otimes B$ is also an algebra in
${\sf  C}$ where
$\eta_{A\otimes B}=\eta_{A}\otimes \eta_{B}$ and $\mu_{A\otimes
	B}=(\mu_{A}\otimes \mu_{B})\circ (A\otimes c_{B,A}\otimes B).$
}
\end{definition}

\begin{definition}
{\rm 
A coalgebra  in ${\sf  C}$ is a triple ${D} = (D,
\varepsilon_{D}, \delta_{D})$ where $D$ is an object in ${\sf
C}$ and $\varepsilon_{D}: D\rightarrow K$ (counit),
$\delta_{D}:D\rightarrow D\otimes D$ (coproduct) are morphisms in
${\sf  C}$ such that $(\varepsilon_{D}\otimes D)\circ
\delta_{D}= id_{D}=(D\otimes \varepsilon_{D})\circ \delta_{D}$,
$(\delta_{D}\otimes D)\circ \delta_{D}=
 (D\otimes \delta_{D})\circ \delta_{D}.$ If ${D} = (D, \varepsilon_{D},
 \delta_{D})$ and
${ E} = (E, \varepsilon_{E}, \delta_{E})$ are coalgebras, a morphism 
$f:D\rightarrow E$ in  {\sf  C} is a coalgebra morphism if $(f\otimes f)\circ
\delta_{D} =\delta_{E}\circ f$, $\varepsilon_{E}\circ f
=\varepsilon_{D}.$ 

Given  $D$, $E$ coalgebras in ${\sf  C}$, the tensor product $D\otimes E$ is a
coalgebra in ${\sf  C}$ where $\varepsilon_{D\otimes
E}=\varepsilon_{D}\otimes \varepsilon_{E}$ and $\delta_{D\otimes
E}=(D\otimes c_{D,E}\otimes E)\circ( \delta_{D}\otimes \delta_{E}).$
}
\end{definition}

\begin{definition}
	{\rm 
 Let ${D} = (D, \varepsilon_{D},
\delta_{D})$ be a coalgebra and $A=(A, \eta_{A}, \mu_{A})$ an
algebra in $\sf{C}$. By ${\mathcal  H}(D,A)$ we denote the set of morphisms
$f:D\rightarrow A$ in ${\sf  C}$. With the convolution operation
$f\ast g= \mu_{A}\circ (f\otimes g)\circ \delta_{D}$, ${\mathcal  H}(D,A)$ is an algebra where the unit element is $\eta_{A}\circ \varepsilon_{D}=\varepsilon_{D}\otimes \eta_{A}$.
}
\end{definition}

\begin{definition}
{\rm 
 Let  $A$ be an algebra. The pair
$(M,\varphi_{M})$ is a left $A$-module if $M$ is an object in
${\sf  C}$ and $\varphi_{M}:A\otimes M\rightarrow M$ is a morphism
in ${\sf  C}$ satisfying $\varphi_{M}\circ(
\eta_{A}\ot M)=id_{M}$, $\varphi_{M}\circ (A\ot \varphi_{M})=\varphi_{M}\circ
(\mu_{A}\ot M)$. Given two left ${A}$-modules $(M,\varphi_{M})$
and $(N,\varphi_{N})$, $f:M\rightarrow N$ is a morphism of left
${A}$-modules if $\varphi_{N}\circ (A\ot f)=f\circ \varphi_{M}$.  

The  composition of morphisms of left $A$-modules is a morphism of left $A$-modules. Then left $A$-modules form a category that we will denote by $\;_{\sf A}${\sf Mod}.

}

\end{definition}

\begin{definition}
{\rm 
We say that $X$ is a
bialgebra  in ${\sf  C}$ if $(X, \eta_{X}, \mu_{X})$ is an
algebra, $(X, \varepsilon_{X}, \delta_{X})$ is a coalgebra, and
$\varepsilon_{X}$ and $\delta_{X}$ are algebra morphisms
(equivalently, $\eta_{X}$ and $\mu_{X}$ are coalgebra morphisms). Moreover, if there exists a morphism $\lambda_{X}:X\rightarrow X$ in ${\sf  C}$,
called the antipode of $X$, satisfying that $\lambda_{X}$ is the inverse of $id_{X}$ in ${\mathcal  H}(X,X)$, i.e., 
\begin{equation}
\label{antipode}
id_{X}\ast \lambda_{X}= \eta_{X}\circ \varepsilon_{X}= \lambda_{X}\ast id_{X},
\end{equation}
we say that $X$ is a Hopf algebra. A morphism of Hopf algebras is an algebra-coalgebra morphism. Note that, if $f:X\rightarrow Y$ is a Hopf algebra morphism the following equality holds:
\begin{equation}
\label{morant}
\lambda_{Y}\co f=f\co \lambda_{X}.
\end{equation} 

With the composition of morphisms in {\sf C} we can define a category whose objects are Hopf algebras  and whose morphisms are morphisms of Hopf algebras. We denote this category by ${\sf  Hopf}$.

Note that if $X=(X,\eta_{X},\mu_{X},\varepsilon_{X},\delta_{X},\lambda_{X})$ is a Hopf algebra in $\sf{C}$ such that its antipode, $\lambda_{X}$, is an isomorphism, then $X^{cop}=(X,\eta_{X},\mu_{X},\varepsilon_{X},c_{X,X}^{-1}\circ\delta_{X},\lambda_{X}^{-1})$ is a Hopf algebra in $\overline{\sf{C}}$ (see \cite{MAJ}).

A Hopf algebra is commutative if $\mu_{X}\co c_{X,X}=\mu_{X}$ and cocommutative if $c_{X,X}\co \delta_{X}=\delta_{X}.$ It is easy to see that in both cases $\lambda_{X}\circ \lambda_{X} =id_{X}$.
}
\end{definition}

If $X$ is a Hopf algebra, relevant properties of its antipode, $\lambda_{X}$, are the following:  It is antimultiplicative and anticomultiplicative 
\begin{equation}
\label{a-antip}
\lambda_{X}\co \mu_{X}=  \mu_{X}\co (\lambda_{X}\ot \lambda_{X})\co c_{X,X},\;\;\;\; \delta_{X}\co \lambda_{X}=c_{X,X}\co (\lambda_{X}\ot \lambda_{X})\co \delta_{X}, 
\end{equation}
and leaves the unit and counit invariant, i.e., 
\begin{equation}
\label{u-antip}
\lambda_{X}\co \eta_{X}=  \eta_{X},\;\; \varepsilon_{X}\co \lambda_{X}=\varepsilon_{X}.
\end{equation}
So, it is a direct consequence of these identities that, if $X$ is commutative, then $\lambda_{X}$ is an algebra morphism and, if $X$ is cocommutative, then $\lambda_{X}$ is a coalgebra morphism.

In the following definition we recall the notion of left module (co)algebra.

\begin{definition}
{\rm 
Let $X$ be a Hopf algebra. An algebra $A$  is said to be a left $X$-module algebra if $(A, \varphi_{A})$ is a left $X$-module and $\eta_{A}$, $\mu_{A}$ are morphisms of left $X$-modules, i.e.,
\begin{equation}
\label{mod-alg}
\varphi_{A}\circ (X\otimes \eta_{A})=\varepsilon_{X}\otimes \eta_{A},\;\;\varphi_{A}\circ (X\otimes \mu_{A})=\mu_{A}\circ \varphi_{A\otimes A},
\end{equation}
where  $\varphi_{A\otimes A}=(\varphi_{A}\otimes \varphi_{A})\circ (X\otimes c_{X,A}\otimes A)\circ (\delta_{X}\otimes A\otimes A)$ is the left action on $A\otimes A$. 
}
\end{definition}
\begin{definition}
Let $X$ be a Hopf algebra. A coalgebra $D$ is said to be a left $X$-module coalgebra if $(D,\varphi_{D})$ is a left $X$-module and $\varepsilon_{D}$, $\delta_{D}$ are morphisms of left $X$-modules, in other words, the following equalities hold:
\begin{equation}\label{mod-coalg}
\varepsilon_{D}\circ\varphi_{D}=\varepsilon_{H}\otimes\varepsilon_{D},\;\;\delta_{D}\circ\varphi_{D}=\varphi_{D\otimes D}\circ(H\otimes\delta_{D}).
\end{equation}
Equivalently, $(D,\varphi_{D})$ is a left $X$-module coalgebra if and only if $\varphi_{D}$ is a coalgebra morphism.
\end{definition}

The following result will be interesting along this paper.
\begin{theorem}\label{th.interest}
Let $X=(X,\eta_{X},\mu_{X},\varepsilon_{X},\delta_{X},\lambda_{X})$ and $H=(H,\eta_{H},\mu_{H},\varepsilon_{H},\delta_{H},\lambda_{H})$ be Hopf algebras in $\sf{C}$ such that there exists a morphism $\varphi_{H}\colon X\otimes H\rightarrow H$ satisfying the following conditions:
\begin{itemize}
\item[(i)] $\varphi_{H}\circ(X\otimes\mu_{H})=\mu_{H}\circ(\varphi_{H}\otimes\varphi_{H})\circ(X\otimes c_{X,H}\otimes H)\circ(\delta_{X}\otimes H\otimes H),$
\item[(ii)] $\varphi_{H}$ is a coalgebra morphism.
\end{itemize}
Then, $\varphi_{H}\circ(X\otimes\eta_{H})=\varepsilon_{X}\otimes\eta_{H}$ holds.
\end{theorem}
\begin{proof} The equality follows by:
\begin{align*}
&\varphi_{H}\circ(X\otimes\eta_{H})\\=&(\varphi_{H}\otimes(\varepsilon_{H}\circ\varphi_{H}))\circ(X\otimes c_{X,H}\otimes H)\circ(\delta_{X}\otimes \eta_{H}\otimes\eta_{H})\;\footnotesize\textnormal{(by (ii), naturality of $c$ and (co)unit properties)}\\=&\mu_{H}\circ(\varphi_{H}\otimes(\eta_{H}\circ\varepsilon_{H}\circ\varphi_{H}))\circ(X\otimes c_{X,H}\otimes H)\circ(\delta_{X}\otimes\eta_{H}\otimes\eta_{H})\;\footnotesize\textnormal{(by unit property)}\\=&\mu_{H}\circ(\varphi_{H}\otimes((id_{H}\ast \lambda_{H})\circ\varphi_{H}))\circ(X\otimes c_{X,H}\otimes H)\circ(\delta_{X}\otimes\eta_{H}\otimes\eta_{H})\;\footnotesize\textnormal{(by \eqref{antipode})}\\=&\mu_{H}\circ(\varphi_{H}\otimes(\mu_{H}\circ(H\otimes\lambda_{H})\circ(\varphi_{H}\otimes\varphi_{H})\circ(X\otimes c_{X,H}\otimes H)\circ(\delta_{X}\otimes\delta_{H})))\circ(X\otimes c_{X,H}\otimes H)\\&\circ(\delta_{X}\otimes\eta_{H}\otimes\eta_{H})\;\footnotesize\textnormal{(by (ii))}\\=&\mu_{H}\circ((\mu_{H}\circ(\varphi_{H}\otimes\varphi_{H})\circ(X\otimes c_{X,H}\otimes H)\circ(\delta_{X}\otimes H\otimes H))\otimes(\lambda_{H}\circ\varphi_{H}))\\&\circ(X\otimes((H\otimes c_{X,H}\otimes H)\circ(c_{X,H}\otimes\delta_{H})))\circ(\delta_{X}\otimes\eta_{H}\otimes\eta_{H})\;\footnotesize\textnormal{(by naturality of $c$, coassociativity of $\delta_{X}$}\\&\footnotesize\textnormal{and associativity of $\mu_{H}$)}\\=&\mu_{H}\circ((\varphi_{H}\circ(X\otimes\mu_{H}))\otimes(\lambda_{H}\circ\varphi_{H}))\circ(X\otimes((H\otimes c_{X,H}\otimes H)\circ(c_{X,H}\otimes\delta_{H})))\circ(\delta_{X}\otimes\eta_{H}\otimes\eta_{H})\\&\footnotesize\textnormal{(by (i))}\\=&\mu_{H}\circ(H\otimes\lambda_{H})\circ(\varphi_{H}\otimes\varphi_{H})\circ(X\otimes c_{X,H}\otimes H)\circ(\delta_{X}\otimes(\delta_{H}\circ\eta_{H}))\;\footnotesize\textnormal{(by naturality of $c$ and unit property)}\\=&(id_{H}\ast\lambda_{H})\circ\varphi_{H}\circ(X\otimes\eta_{H})\;\footnotesize\textnormal{(by (ii))}\\=&\eta_{H}\circ\varepsilon_{H}\circ\varphi_{H}\circ(X\otimes\eta_{H})\;\footnotesize\textnormal{(by \eqref{antipode})}\\=&\varepsilon_{X}\otimes\eta_{H}\;\footnotesize\textnormal{(by (ii) and (co)unit properties)}.\qedhere
\end{align*}
\end{proof}
\begin{corollary}
Let $X=(X,\eta_{X},\mu_{X},\varepsilon_{X},\delta_{X},\lambda_{X})$ and $H=(H,\eta_{H},\mu_{H},\varepsilon_{H},\delta_{H},\lambda_{H})$ be Hopf algebras in $\sf{C}$. If $(H,\varphi_{H})$ is a left $X$-module coalgebra and $\mu_{H}$ is a morphism of left $X$-modules, then $(H,\varphi_{H})$ is a left $X$-module algebra.
\end{corollary}
In the braided setting the definition of Hopf brace is the following:

\begin{definition}
\label{H-brace}
{\rm Let $H=(H, \varepsilon_{H}, \delta_{H})$ be a coalgebra in {\sf C}. Let's assume that there are two algebra structures $(H, \eta_{H}^1, \mu_{H}^1)$, $(H, \eta_{H}^2, \mu_{H}^2)$ defined on $H$ and suppose that there exist two endomorphism of $H$ denoted by $\lambda_{H}^{1}$ and $\lambda_{H}^{2}$. We will say that 
$$(H, \eta_{H}^{1}, \mu_{H}^{1}, \eta_{H}^{2}, \mu_{H}^{2}, \varepsilon_{H}, \delta_{H}, \lambda_{H}^{1}, \lambda_{H}^{2})$$
is a Hopf brace in {\sf C} if:
\begin{itemize}
\item[(i)] $H_{1}=(H, \eta_{H}^{1}, \mu_{H}^{1},  \varepsilon_{H}, \delta_{H}, \lambda_{H}^{1})$ is a Hopf algebra in {\sf C}.
\item[(ii)] $H_{2}=(H, \eta_{H}^{2}, \mu_{H}^{2},  \varepsilon_{H}, \delta_{H}, \lambda_{H}^{2})$ is a Hopf algebra in {\sf C}.
\item[(iii)] The  following equality holds:
$$\mu_{H}^{2}\co (H\ot \mu_{H}^{1})=\mu_{H}^{1}\co (\mu_{H}^{2}\ot \Gamma_{H_{1}} )\co (H\ot c_{H,H}\ot H)\co (\delta_{H}\ot H\ot H),$$
\end{itemize}
where  $$\Gamma_{H_{1}}\coloneqq\mu_{H}^{1}\co (\lambda_{H}^{1}\ot \mu_{H}^{2})\co (\delta_{H}\ot H).$$

Following \cite{RGON}, a Hopf brace will be denoted by ${\mathbb H}=(H_{1}, H_{2})$ or in a simpler way by $\mathbb{H}$.

}
\end{definition}

\begin{definition}
 If  ${\mathbb H}$ is a Hopf brace in {\sf C}, we will say that ${\mathbb H}$ is cocommutative if $\delta_{H}=c_{H,H}\circ \delta_{H}$, i.e., if $H_{1}$ and $H_{2}$ are cocommutative Hopf algebras in {\sf C}.
 
Note that by \cite[Corollary 5]{Sch}, if $H$ is a  cocommutative Hopf algebra  in the  braided monoidal category {\sf C}, the identity 
\begin{equation}
\label{ccb}
c_{H,H}\circ c_{H,H}=id_{H\otimes H} 
\end{equation}
holds.
\end{definition}

\begin{definition}
{\rm  Given two Hopf braces ${\mathbb H}$  and  ${\mathbb B}$ in {\sf C}, a morphism $f$ in {\sf C} between the two underlying objects is called a morphism of Hopf braces if both $f:H_{1}\rightarrow B_{1}$ and $f:H_{2}\rightarrow B_{2}$ are Hopf algebra morphisms.
		
Hopf braces together with morphisms of Hopf braces form a category which we denote by {\sf HBr}. Moreover, cocommutative Hopf braces constitute a full subcategory of $\sf{HBr}$ which we will denote by $\sf{cocHBr}$.
		
}
\end{definition}
Let ${\mathbb H}$  be a Hopf brace in {\sf C}. Then 
\begin{equation}
\eta_{H}^{1}=\eta_{H}^2, 
 \end{equation}
holds and, by \cite[Lemma 1.7]{AGV},  in this braided setting  the equality
\begin{equation}
\label{agv1}
\Gamma_{H_{1}}\circ (H\otimes \lambda_{H}^1)=\mu_{H}^{1}\circ ((\lambda_{H}^1\circ \mu_{H}^{2})\otimes H)\circ (H\otimes c_{H,H}) \circ (\delta_{H}\otimes H)
\end{equation}
also holds. Moreover, in our braided context \cite[Lemma 1.8]{AGV} and \cite[Remark 1.9]{AGV} hold and then we have that the algebra $(H,\eta_{H}^{1}, \mu_{H}^{1})$ is a left $H_{2}$-module algebra with action $\Gamma_{H_{1}}$ and $\mu_{H}^2$ admits the following expression:
\begin{equation}
\label{eb2}
\mu_{H}^2=\mu_{H}^{1}\circ (H\otimes \Gamma_{H_{1}})\circ (\delta_{H}\otimes H). 
\end{equation}
In addition, by \cite[Lemma 2.2]{AGV}, $\Gamma_{H_{1}}$ is a coalgebra morphism when $\mathbb{H}$ is cocommutative.

\begin{lemma}
Let $\mathbb{H}$ be a Hopf brace in $\sf{C}$. The equality
\begin{equation}\label{GammaL2}
\Gamma_{H_{1}}\circ(H\otimes\lambda_{H}^{2})\circ\delta_{H}=\lambda_{H}^{1}
\end{equation}
holds.
\end{lemma}
\begin{proof}The equality follows by:
\begin{align*} 
&\Gamma_{H_{1}}\circ(H\otimes\lambda_{H}^{2})\circ\delta_{H}\\=&\mu_{H}^{1}\circ(\lambda_{H}^{1}\otimes\mu_{H}^{2})\circ(\delta_{H}\otimes\lambda_{H}^{2})\circ\delta_{H}\;\footnotesize\textnormal{(by definition of $\Gamma_{H_{1}}$)}\\=&\mu_{H}^{1}\circ(\lambda_{H}^{1}\otimes(id_{H}\ast\lambda_{H}^{2}))\circ\delta_{H}\;\footnotesize\textnormal{(by coassociativity of $\delta_{H}$)}\\=&\lambda_{H}^{1}\;\footnotesize\textnormal{(by \eqref{antipode} and (co)unit properties)}.\qedhere
\end{align*}
\end{proof}

To conclude this introductory section we will remember the notion of finite object in $\sf{C}$, since they are going to be of special interest throughout Section \ref{sec3}.
\begin{definition}\label{finite}
An object $P$ in $\sf{C}$ is finite if there exists an object $P^{\ast}$, called the dual of $P$, and a $\sf{C}$-adjunction $P\otimes -\dashv P^{\ast}\otimes -$ between the tensor functors.
\end{definition}
We will denote by $a_{P}$ and $b_{P}$ the unit and the counit of the previous $\sf{C}$-adjunction, respectively. Finite objects in $\sf{C}$ constitute a full subcategory of $\sf{C}$ that we will denote by $\sf{C^{f}}$. Note that, for every finite object $P$ in $\sf{C}$, we have a natural algebra structure in $\sf{C}$ over the tensor object $P^{\ast}\otimes P$ as we can see in the following lemma, whose proof is straightforward.
\begin{lemma}
Let $P$ be a finite object in $\sf{C}$, then $P^{\ast}\otimes P$ is an algebra in $\sf{C}$ with product and unit given by
\[\mu_{P^{\ast}\otimes P}\coloneqq P^{\ast}\otimes b_{P}(K)\otimes P\]and\[\eta_{P^{\ast}\otimes P}\coloneqq a_{P}(K),\] respectively.
\end{lemma}
Moreover, it is going to be useful the following lemma.
\begin{lemma}
If $P$ is a finite object in $\sf{C}$, then 
\begin{equation}\label{finitebraid}
(c_{P,P^{\ast}}\otimes P)\circ(P\otimes a_{P}(K))=(P^{\ast}\otimes c_{P,P}^{-1})\circ(a_{P}(K)\otimes P).
\end{equation}
\end{lemma}
\begin{proof}
The equality \eqref{finitebraid} follows by:
\begin{align*}
&(c_{P,P^{\ast}}\otimes P)\circ(P\otimes a_{P}(K))\\=&(P^{\ast}\otimes(c_{P,P}^{-1}\circ c_{P,P}))\circ(c_{P,P^{\ast}}\otimes P)\circ(P\otimes a_{P}(K))\;\footnotesize\textnormal{(by the isomorphism condition for $c_{P,P}$)}\\=&(P^{\ast}\otimes c_{P,P}^{-1})\circ(a_{P}(K)\otimes P)\;\footnotesize\textnormal{(by naturality of $c$ and $c_{P,K}=id_{K}$}).\qedhere
\end{align*}
\end{proof}
\section{Brace triples and Hopf braces}\label{sect2}
The aim of this part is to prove that we can characterise Hopf braces in $\sf{C}$ via another structures. This new structures will be known as {brace triples}.
\begin{definition}\label{BTdef}
	Consider $H=(H,\eta_{H},\mu_{H},\varepsilon_{H},\delta_{H},\lambda_{H})$ a Hopf algebra in $\sf{C}$ with $\lambda_{H}$ an isomorphism and let $\gamma_{H}\colon H\otimes H\rightarrow H$ and $T_{H}\colon H\rightarrow H$ be morphisms in $\sf{C}.$ We will say that $(H,\gamma_{H},T_{H})$ is a brace triple if the following conditions hold:
	\begin{itemize}
		\item[(i)] $(\gamma_{H}\otimes H)\circ(H\otimes c_{H,H})\circ(\delta_{H}\otimes H)=(\gamma_{H}\otimes H)\circ(H\otimes c_{H,H})\circ((c_{H,H}\circ\delta_{H})\otimes H).$
		\item[(ii)] $\gamma_{H}$ is a coalgebra morphism, i.e.:
		\begin{itemize}
			\item[(ii.1)] $\delta_{H}\circ\gamma_{H}=(\gamma_{H}\otimes\gamma_{H})\circ(H\otimes c_{H,H}\otimes H)\circ(\delta_{H}\otimes\delta_{H})$,
			 \item[(ii.2)]$\varepsilon_{H}\circ\gamma_{H}=\varepsilon_{H}\otimes\varepsilon_{H}$.
		\end{itemize} 
		\item[(iii)] $\gamma_{H}\circ(H\otimes\mu_{H})=\mu_{H}\circ(\gamma_{H}\otimes\gamma_{H})\circ(H\otimes c_{H,H}\otimes H)\circ(\delta_{H}\otimes H\otimes H)$.
		\item[(iv)] $\gamma_{H}\circ(H\otimes\gamma_{H})=\gamma_{H}\circ((\mu_{H}\circ(H\otimes\gamma_{H})\circ(\delta_{H}\otimes H))\otimes H)$.
		\item[(v)] $\gamma_{H}\circ(\eta_{H}\otimes H)=id_{H}$.
		\item[(vi)] $T_{H}$ is an isomorphism in $\sf{C}$ such that the following equalities are verified:
		\begin{itemize}
			\item[(vi.1)] $\delta_{H}\circ T_{H}=c_{H,H}\circ(T_{H}\otimes T_{H})\circ\delta_{H}$.
			\item[(vi.2)] $\varepsilon_{H}\circ T_{H}=\varepsilon_{H}$
			\item[(vi.3)] $\mu_{H}\circ(H\otimes\gamma_{H})\circ((\delta_{H}\circ T_{H})\otimes H) =\mu_{H}\circ(H\otimes\gamma_{H})\circ(((T_{H}\otimes T_{H})\circ\delta_{H})\otimes H)$.
			\item[(vi.4)]$\gamma_{H}\circ(H\otimes T_{H})\circ\delta_{H}=\lambda_{H}$.
			\item[(vi.5)] $\gamma_{H}\circ (T_{H}\otimes H)\circ\delta_{H}=\lambda_{H}^{-1}\circ T_{H}$.
		\end{itemize}
	\end{itemize}
\end{definition}
\begin{remark}
Given a brace triple $(H,\gamma_{H},T_{H})$, conditions (ii) and (iii) of Definition \ref{BTdef} imply that 
\begin{equation}\label{olde6}
\gamma_{H}\circ(H\otimes\eta_{H})=\varepsilon_{H}\otimes \eta_{H}
\end{equation}
holds by Theorem \ref{th.interest}.
\end{remark}
\begin{remark}
Let $(H,\gamma_{H},T_{H})$ be a brace triple. Note that condition (vi.5) of Definition \ref{BTdef} is equivalent to 
\begin{equation}\label{olde75}
\gamma_{H}\circ(H\otimes T_{H}^{-1})\circ c_{H,H}^{-1}\circ\delta_{H}=\lambda_{H}^{-1}.
\end{equation}
In fact, on the one side, suppose that (vi.5) of Definition \ref{BTdef} holds. Then, we have that:
\begin{align*}
&\lambda_{H}^{-1}\\=&\lambda_{H}^{-1}\circ T_{H}\circ T_{H}^{-1}\;\footnotesize\textnormal{(by the condition of isomorphism for $T_{H}$)}\\=&\gamma_{H}\circ(T_{H}\otimes H)\circ\delta_{H}\circ T_{H}^{-1}\;\footnotesize\textnormal{(by (vi.5) of Definition \ref{BTdef})}\\=&\gamma_{H}\circ(T_{H}\otimes H)\circ(T_{H}^{-1}\otimes T_{H}^{-1})\circ c_{H,H}^{-1}\circ\delta_{H}\;\footnotesize\textnormal{(by (vi.1) of Definition \ref{BTdef} and the isomorphism condition}\\&\footnotesize\textnormal{for $T_{H}$ and $c_{H,H}$)}\\=&\gamma_{H}\circ(H\otimes T_{H}^{-1})\circ c_{H,H}^{-1}\circ\delta_{H}\;\footnotesize\textnormal{(by the condition of isomorphism for $T_{H}$)}.
\end{align*}
On the other side, suppose now that \eqref{olde75} holds. Then,
\begin{align*}
&\lambda_{H}^{-1}\circ T_{H}\\=&\gamma_{H}\circ(H\otimes T_{H}^{-1})\circ c_{H,H}^{-1}\circ\delta_{H}\circ T_{H}\;\footnotesize\textnormal{(by \eqref{olde75})}\\=&\gamma_{H}\circ(H\otimes T_{H}^{-1})\circ c_{H,H}^{-1}\circ c_{H,H}\circ (T_{H}\otimes T_{H})\circ\delta_{H}\;\footnotesize\textnormal{(by (vi.1) of Definition \ref{BTdef})}\\=&\gamma_{H}\circ(T_{H}\otimes H)\circ\delta_{H}\;\footnotesize\textnormal{(by the isomorphism condition for $c_{H,H}$ and $T_{H}$)}.
\end{align*}
\end{remark}
\begin{definition}
	Let $(H,\gamma_{H},T_{H})$ and $(B,\gamma_{B}, T_{B})$ be brace triples and $f\colon H\rightarrow B$ a morphism in $\sf{C}$. We will say that $f$ is a morphism of brace triples if $f$ is a Hopf algebra morphism and
		\begin{equation}\label{CondMorBT} f\circ\gamma_{H}=\gamma_{B}\circ(f\otimes f)
		\end{equation}
	holds.
\end{definition}
Brace triples and their morphisms form a category which we will denote by $\sf{BT}$. 

\begin{remark}\label{TB-cocTB}
Suppose that $(H,\gamma_{H},T_{H})$ is a brace triple with $H$ cocommutative. Under this condition, note that (i) of Definition \ref{BTdef} always holds and take also into account that (vi.1) becomes $\delta_{H}\circ T_{H}=(T_{H}\otimes T_{H})\circ\delta_{H}$. This implies that (vi.3) always holds in the cocommutative setting. Moreover, $\lambda_{H}^{-1}=\lambda_{H}$  (due to $\lambda_{H}\circ\lambda_{H}=id_{H}$), so this implies that (vi.5) becomes $\gamma_{H}\circ(T_{H}\otimes H)\circ\delta_{H}=\lambda_{H}\circ T_{H}$. Therefore, as a consequence of (vi.4) and (vi.5), we obtain that $$\gamma_{H}\circ (H\otimes T_{H})\circ\delta_{H}=\lambda_{H}=\gamma_{H}\circ(T_{H}\otimes H)\circ\delta_{H}\circ T_{H}^{-1}.$$

Cocommutative brace triples constitute a full subcategory of $\sf{BT}$ which we will denote by $\sf{cocBT}.$
\end{remark}
\begin{definition}\label{sHbr}
Let $\mathbb{H}$ be a Hopf brace in $\sf{C}$. We will say that $\mathbb{H}$ is a \texttt{s}-Hopf brace if the following conditions hold:
\begin{itemize}
\item[(i)] $(\Gamma_{H_{1}}\otimes H)\circ(H\otimes c_{H,H})\circ(\delta_{H}\otimes H)=(\Gamma_{H_{1}}\otimes H)\circ(H\otimes c_{H,H})\circ((c_{H,H}\circ\delta_{H})\otimes H)$.
\item[(ii)] $\lambda_{H}^{1}$ and $\lambda_{H}^{2}$ are isomorphisms in $\sf{C}$ such that the following conditions hold:
\begin{itemize}
\item[(ii.1)] $\mu_{H}^{1}\circ(H\otimes\Gamma_{H_{1}})\circ((\delta_{H}\circ\lambda_{H}^{2})\otimes H)=\mu_{H}^{1}\circ(H\otimes\Gamma_{H_{1}})\circ(((\lambda_{H}^{2}\otimes\lambda_{H}^{2})\circ\delta_{H})\otimes H)$.
\item[(ii.2)] $\Gamma_{H_{1}}\circ(\lambda_{H}^{2}\otimes H)\circ\delta_{H}=(\lambda_{H}^{1})^{-1}\circ\lambda_{H}^{2}.$
\end{itemize} 
\end{itemize}
\end{definition}
With the obvious morphisms, \texttt{s}-Hopf braces constitute a full subcategory of $\sf{HBr}$, and we will denote it by \texttt{s}$\sf{HBr}$.
\begin{remark}\label{remark cocom class GH1}
Let's assume that $\sf{C}$ is symmetric. Under this assumption, condition (i) of Definition \ref{sHbr} means that $(H_{1},\Gamma_{H_{1}})$ is in the cocommutativity class of $H_{2}$ following the notion introduced in \cite[Definition 2.1 and Definition 2.2]{CCH}.
\end{remark}
\begin{remark}
Note that, for a \texttt{s}-Hopf brace $\mathbb{H}$, condition (ii.2) of Definition \ref{sHbr} is equivalent to 
\begin{equation}\label{oldz22}
\Gamma_{H_{1}}\circ(H\otimes (\lambda_{H}^{2})^{-1})\circ c_{H,H}^{-1}\circ\delta_{H}=(\lambda_{H}^{1})^{-1}.
\end{equation}
Indeed, suppose that (ii.2) of Definition \ref{sHbr} holds, then:
\begin{align*}
&(\lambda_{H}^{1})^{-1}\\=&(\lambda_{H}^{1})^{-1}\circ\lambda_{H}^{2}\circ(\lambda_{H}^{2})^{-1}\;\footnotesize\textnormal{(by the condition of isomorphism for $\lambda_{H}^{2}$)}\\=&\Gamma_{H_{1}}\circ(\lambda_{H}^{2}\otimes H)\circ\delta_{H}\circ(\lambda_{H}^{2})^{-1}\;\footnotesize\textnormal{(by (ii.2) of Definition \ref{sHbr})}\\=&\Gamma_{H_{1}}\circ(\lambda_{H}^{2}\otimes H)\circ((\lambda_{H}^{2})^{-1}\otimes(\lambda_{H}^{2})^{-1})\circ c_{H,H}^{-1}\circ\delta_{H}\;\footnotesize\textnormal{(by \eqref{a-antip} and the isomorphism condition}\\&\footnotesize\textnormal{for $\lambda_{H}^{2}$ and $c_{H,H}$)}\\=&\Gamma_{H_{1}}\circ(H\otimes(\lambda_{H}^{2})^{-1})\circ c_{H,H}^{-1}\circ\delta_{H}\;\footnotesize\textnormal{(by the condition of isomorphism for $\lambda_{H}^{2}$)}.
\end{align*}
On the other hand, we have that:
\begin{align*}
&(\lambda_{H}^{1})^{-1}\circ\lambda_{H}^{2}\\=&\Gamma_{H_{1}}\circ(H\otimes(\lambda_{H}^{2})^{-1})\circ c_{H,H}^{-1}\circ\delta_{H}\circ\lambda_{H}^{2}\;\footnotesize\textnormal{(by \eqref{oldz22})}\\=&\Gamma_{H_{1}}\circ(H\otimes(\lambda_{H}^{2})^{-1})\circ c_{H,H}^{-1}\circ c_{H,H}\circ(\lambda_{H}^{2}\otimes\lambda_{H}^{2})\circ\delta_{H}\;\footnotesize\textnormal{(by \eqref{a-antip})}\\=&\Gamma_{H_{1}}\circ(\lambda_{H}^{2}\otimes H)\circ\delta_{H}\;\footnotesize\textnormal{(by the isomorphism condition for $\lambda_{H}^{2}$ and $c_{H,H}$)}.
\end{align*}
\end{remark}
\begin{remark}\label{sHB-cocHB} Consider $\mathbb{H}$ a cocommutative \texttt{s}-Hopf brace. Under cocommutativity assumption, note that (i) of Definition \ref{sHbr} always holds. In addition, $\delta_{H}\circ\lambda_{H}^{2}=(\lambda_{H}^{2}\otimes\lambda_{H}^{2})\circ\delta_{H}$, so (ii.1) always holds too. Moreover, under cocommutativity conditions, $\lambda_{H}^{k}$ is an isomorphism with inverse itself for all $k=1,2$. Therefore, condition (ii.2) of Definition \ref{sHbr} is satisfied. Indeed:
\begin{align*}
&\Gamma_{H_{1}}\circ(\lambda_{H}^{2}\otimes H)\circ\delta_{H}\\=&\mu_{H}^{1}\circ(\lambda_{H}^{1}\otimes\mu_{H}^{2})\circ((\delta_{H}\circ\lambda_{H}^{2})\otimes H)\circ\delta_{H}\;\footnotesize\textnormal{(by definition of $\Gamma_{H_{1}}$)}\\=&\mu_{H}^{1}\circ(\lambda_{H}^{1}\otimes\mu_{H}^{2})\circ(((\lambda_{H}^{2}\otimes\lambda_{H}^{2})\circ\delta_{H})\otimes H)\circ\delta_{H}\;\footnotesize\textnormal{(by \eqref{a-antip} and cocommutativity of $\delta_{H}$)}\\=&\mu_{H}^{1}\circ((\lambda_{H}^{1}\circ\lambda_{H}^{2})\otimes (\lambda_{H}^{2}\ast id_{H}))\circ\delta_{H}\;\footnotesize\textnormal{(by coassociativity of $\delta_{H}$)}\\=&\mu_{H}^{1}\circ((\lambda_{H}^{1}\circ\lambda_{H}^{2})\otimes (\eta_{H}\circ\varepsilon_{H}))\circ\delta_{H}\;\footnotesize\textnormal{(by \eqref{antipode})}\\=&\lambda_{H}^{1}\circ\lambda_{H}^{2}\;\footnotesize\textnormal{(by (co)unit property)}.
\end{align*}
So, under cocommutativity supposition, every Hopf brace is a \texttt{s}-Hopf brace, that is to say, $\texttt{s}\sf{HBr}=\sf{cocHBr}$.
\end{remark}

In this first result, we will prove that every brace triple induce a \texttt{s}-Hopf brace in $\sf{C}$.
\begin{theorem}\label{Prop Brace TB}
	Let $(H,\gamma_{H},T_{H})$ be a brace triple in $\sf{C}$. Then $\mathbb{H}_{\sf{BT}}=(H,H_{\sf{BT}})$
	is a \texttt{s}-Hopf brace in $\sf{C}$ being $H_{\sf{BT}}$ the Hopf algebra structure defined by $H_{\sf{BT}}= (H,\eta_{H},\mu_{H}^{\scalebox{0.6}{\sf{BT}}},\varepsilon_{H},\delta_{H},T_{H}),$
	where $\mu_{H}^{\scalebox{0.6}{\sf{BT}}}\coloneqq\mu_{H}\circ(H\otimes\gamma_{H})\circ(\delta_{H}\otimes H)$.
\end{theorem}
	\begin{proof}
		At first we will prove that $H_{\sf{BT}}$ is a Hopf algebra in $\sf{C}$. Note that we already know that $(H,\varepsilon_{H},\delta_{H})$ is a coalgebra in $\sf{C}$ and that $\eta_{H}$ is a coalgebra morphism. We begin by proving the unit property for $\mu_{H}^{\scalebox{0.6}{\sf{BT}}}$. Indeed, on the one side,
		\begin{align*}
			&\mu_{H}^{\scalebox{0.6}{\sf{BT}}}\circ(\eta_{H}\otimes H)\\=&\mu_{H}\circ(H\otimes\gamma_{H})\circ((\delta_{H}\circ\eta_{H})\otimes H)\;\footnotesize\textnormal{(by definition of $\mu_{H}^{\scalebox{0.6}{\sf{BT}}}$)}\\=&\mu_{H}\circ(\eta_{H}\otimes(\gamma_{H}\circ(\eta_{H}\otimes H)))\;\footnotesize\textnormal{(by the condition of coalgebra morphism for $\eta_{H}$)}\\=&\mu_{H}\circ(\eta_{H}\otimes H)\;\footnotesize\textnormal{(by (v) of Definition \ref{BTdef})}\\=&id_{H}\;\footnotesize\textnormal{(by unit property)},
		\end{align*}
		and, on the other side,
		\begin{align*}
			&\mu_{H}^{\scalebox{0.6}{\sf{BT}}}\circ(H\otimes\eta_{H})\\=&\mu_{H}\circ(H\otimes(\gamma_{H}\circ(H\otimes\eta_{H})))\circ\delta_{H}\;\footnotesize\textnormal{(by definition of $\mu_{H}^{\scalebox{0.6}{\sf{BT}}}$)}\\=&\mu_{H}\circ(H\otimes\varepsilon_{H}\otimes H)\circ(\delta_{H}\otimes\eta_{H})\;\footnotesize\textnormal{(by \eqref{olde6})}\\=&id_{H}\;\footnotesize\textnormal{(by (co)unit property)}.
		\end{align*}
The associativity of $\mu_{H}^{\scalebox{0.6}{\sf{BT}}}$ follows by
		\begin{align*}
			&\mu_{H}^{\scalebox{0.6}{\sf{BT}}}\circ(\mu_{H}^{\scalebox{0.6}{\sf{BT}}}\otimes H)\\=&\mu_{H}\circ(H\otimes \gamma_{H})\circ((\delta_{H}\circ\mu_{H})\otimes H)\circ(H\otimes\gamma_{H}\otimes H)\circ(\delta_{H}\otimes H\otimes H)
			\;\footnotesize\textnormal{(by definition of $\mu_{H}^{\scalebox{0.6}{\sf{BT}}}$)}\\=&\mu_{H}\circ(H\otimes\gamma_{H})\circ(((\mu_{H}\otimes\mu_{H})\circ(H\otimes c_{H,H}\otimes H)\circ(\delta_{H}\otimes(\delta_{H}\circ\gamma_{H})))\otimes H)\circ(\delta_{H}\otimes H\otimes H)\\&\footnotesize\textnormal{(by the condition of coalgebra morphism for $\mu_{H}$)}\\=& \mu_{H}\circ(H\otimes\gamma_{H})\circ(((\mu_{H}\otimes\mu_{H})\circ(H\otimes c_{H,H}\otimes H))\otimes H)\circ(\delta_{H}\otimes((\gamma_{H}\otimes\gamma_{H})\circ(H\otimes c_{H,H}\otimes H)\\&\circ(\delta_{H}\otimes\delta_{H})))\otimes H)\circ(\delta_{H}\otimes H\otimes H)\;\footnotesize\textnormal{(by (ii.1) of Definition \ref{BTdef})}\\=& \mu_{H}\circ(\mu_{H}\otimes(\gamma_{H}\circ(\mu_{H}\otimes H)))\circ(H\otimes((\gamma_{H}\otimes H)\circ(H\otimes c_{H,H})\circ((c_{H,H}\circ\delta_{H})\otimes H))\otimes\gamma_{H}\otimes H)\\&\circ(H\otimes H\otimes c_{H,H}\otimes H\otimes H)\circ(((H\otimes\delta_{H})\circ\delta_{H})\otimes\delta_{H}\otimes H)\;\footnotesize\textnormal{(by naturality of $c$ and coassociativity of $\delta_{H}$)}\\=& \mu_{H}\circ(\mu_{H}\otimes(\gamma_{H}\circ(\mu_{H}\otimes H)))\circ(H\otimes((\gamma_{H}\otimes H)\circ(H\otimes c_{H,H})\circ(\delta_{H}\otimes H))\otimes\gamma_{H}\otimes H)\\&\circ(H\otimes H\otimes c_{H,H}\otimes H\otimes H)\circ(((H\otimes\delta_{H})\circ\delta_{H})\otimes\delta_{H}\otimes H)\;\footnotesize\textnormal{(by (i) of Definition \ref{BTdef})}\\=& \mu_{H}\circ((\mu_{H}\circ(H\otimes\gamma_{H}))\otimes(\gamma_{H}\circ((\mu_{H}\circ(H\otimes\gamma_{H})\circ(\delta_{H}\otimes H))\otimes H)))\circ(H\otimes((H\otimes c_{H,H})\\&\circ(\delta_{H}\otimes H))\otimes H\otimes H)\circ(\delta_{H}\otimes\delta_{H}\otimes H)\;\footnotesize\textnormal{(by naturality of $c$ and coassociativity of $\delta_{H}$)}\\=& \mu_{H}\circ(H\otimes(\mu_{H}\circ(\gamma_{H}\otimes\gamma_{H})\circ(H\otimes c_{H,H}\otimes H)\circ(\delta_{H}\otimes H\otimes H)))\circ(H\otimes H\otimes H\otimes \gamma_{H})\\&\circ(\delta_{H}\otimes\delta_{H}\otimes H)\;\footnotesize\textnormal{(by (iv) of Definition \ref{BTdef} and associativity of $\mu_{H}$)}\\=& \mu_{H}\circ(H\otimes\gamma_{H})\circ(\delta_{H}\otimes(\mu_{H}\circ(H\otimes \gamma_{H})\circ(\delta_{H}\otimes H))))\;\footnotesize\textnormal{(by (iii) of Definition \ref{BTdef})}\\=& \mu_{H}^{\scalebox{0.6}{\sf{BT}}}\circ(H\otimes\mu_{H}^{\scalebox{0.6}{\sf{BT}}})\;\footnotesize\textnormal{(by definition of $\mu_{H}^{\scalebox{0.6}{\sf{BT}}}$)}.
		\end{align*}

Also, $\mu_{H}^{\scalebox{0.6}{\sf{BT}}}$ is a coalgebra morphism. On the one hand, by the condition of coalgebra morphism for $\mu_{H}$, (ii.2) of Definition \ref{BTdef} and the counit property, it is straightforward to compute that $\varepsilon_{H}\circ\mu_{H}^{\scalebox{0.6}{\sf{BT}}}=\varepsilon_{H}\otimes\varepsilon_{H}$ and, on the other hand,
		\begin{align*}
			&\delta_{H}\circ\mu_{H}^{\scalebox{0.6}{\sf{BT}}}\\=& \delta_{H}\circ\mu_{H}\circ(H\otimes\gamma_{H})\circ(\delta_{H}\otimes H)\;\footnotesize\textnormal{(by definition of $\mu_{H}^{\scalebox{0.6}{\sf{BT}}}$)}\\=& (\mu_{H}\otimes\mu_{H})\circ(H\otimes c_{H,H}\otimes H)\circ(\delta_{H}\otimes(\delta_{H}\circ\gamma_{H}))\circ(\delta_{H}\otimes H)\;\footnotesize\textnormal{(by the condition of coalgebra morphism}\\&\footnotesize\textnormal{for $\mu_{H}$)}\\=& (\mu_{H}\otimes\mu_{H})\circ(H\otimes c_{H,H}\otimes H)\circ(\delta_{H}\otimes((\gamma_{H}\otimes\gamma_{H})\circ(H\otimes c_{H,H}\otimes H)\circ(\delta_{H}\otimes\delta_{H})))\circ(\delta_{H}\otimes H)\\&\footnotesize\textnormal{(by (ii.1) of Definition \ref{BTdef})}\\=& (\mu_{H}\otimes\mu_{H})\circ(H\otimes((\gamma_{H}\otimes H)\circ(H\otimes c_{H,H})\circ((c_{H,H}\circ\delta_{H})\otimes H))\otimes\gamma_{H})\circ(H\otimes H\otimes c_{H,H}\otimes H)\\&\circ(((H\otimes \delta_{H})\circ\delta_{H})\otimes\delta_{H})\;\footnotesize\textnormal{(by naturality of $c$ and coassociativity of $\delta_{H}$)}\\=& (\mu_{H}\otimes\mu_{H})\circ(H\otimes((\gamma_{H}\otimes H)\circ(H\otimes c_{H,H})\circ(\delta_{H}\otimes H))\otimes\gamma_{H})\circ(H\otimes H\otimes c_{H,H}\otimes H)\\&\circ(((H\otimes \delta_{H})\circ\delta_{H})\otimes\delta_{H})\;\footnotesize\textnormal{(by (i) of Definition \ref{BTdef})}\\=& ((\mu_{H}\circ(H\otimes\gamma_{H})\circ(\delta_{H}\otimes H))\otimes(\mu_{H}\circ(H\otimes\gamma_{H})\circ(\delta_{H}\otimes H)))\circ(H\otimes c_{H,H}\otimes H)\circ(\delta_{H}\otimes \delta_{H})\\&\footnotesize\textnormal{(by coassociativity of $\delta_{H}$ and naturality of $c$)}\\=& (\mu_{H}^{\scalebox{0.6}{\sf{BT}}}\otimes\mu_{H}^{\scalebox{0.6}{\sf{BT}}})\circ(H\otimes c_{H,H}\otimes H)\circ(\delta_{H}\otimes\delta_{H})\;\footnotesize\textnormal{(by definition of $\mu_{H}^{\scalebox{0.6}{\sf{BT}}}$)}.
		\end{align*}

So, $H_{\sf{BT}}$ is a bialgebra in $\sf{C}$. From now on, we will denote by $\ast_{\scalebox{0.6}{\sf{BT}}}$ the convolution in $\mathcal{H}(H,H_{\sf{BT}})$. 

The conditions for $T_{H}$ to be the antipode for $H_{\sf{BT}}$ follows from the following facts. First note that
		\begin{align*}
			&id_{H}\ast_{\scalebox{0.6}{\sf{BT}}}  T_{H}\\=& \mu_{H}^{\scalebox{0.6}{\sf{BT}}}\circ(H\otimes T_{H})\circ\delta_{H}\;\footnotesize\textnormal{(by definition of $\ast_{\scalebox{0.6}{\sf{BT}}} $)}\\=& \mu_{H}\circ (H\otimes\gamma_{H})\circ(\delta_{H}\otimes T_{H})\circ\delta_{H}\;\footnotesize\textnormal{(by definition of $\mu_{H}^{\scalebox{0.6}{\sf{BT}}}$)}\\=& \mu_{H}\circ(H\otimes(\gamma_{H}\circ(H\otimes T_{H})\circ\delta_{H}))\circ\delta_{H}\;\footnotesize\textnormal{(by coassociativity of $\delta_{H}$)}\\=& id_{H}\ast\lambda_{H}\;\footnotesize\textnormal{(by (vi.4) of Definition \ref{BTdef})}\\=& \varepsilon_{H}\otimes\eta_{H}\;\footnotesize\textnormal{(by \eqref{antipode})}.
		\end{align*}
On the other side,
		\begin{align*}
			&T_{H}\ast_{\scalebox{0.6}{\sf{BT}}}  id_{H}\\=& \mu_{H}^{\scalebox{0.6}{\sf{BT}}}\circ(T_{H}\otimes H)\circ\delta_{H}\;\footnotesize\textnormal{(by definition of $\ast_{\scalebox{0.6}{\sf{BT}}} $)}\\=& \mu_{H}\circ(H\otimes\gamma_{H})\circ((\delta_{H}\circ T_{H})\otimes H)\circ\delta_{H}\;\footnotesize\textnormal{(by definition of $\mu_{H}^{\scalebox{0.6}{\sf{BT}}}$)}\\=& \mu_{H}\circ(H\otimes \gamma_{H})\circ(((T_{H}\otimes T_{H})\circ\delta_{H})\otimes H)\circ\delta_{H}\;\footnotesize\textnormal{(by (vi.3) of Definition \ref{BTdef})}\\=& \mu_{H}\circ(T_{H}\otimes(\gamma_{H}\circ(T_{H}\otimes H)\circ\delta_{H}))\circ\delta_{H}\;\footnotesize\textnormal{(by coassociativity of $\delta_{H}$)}\\=& \mu_{H}\circ(T_{H}\otimes(\lambda_{H}^{-1}\circ T_{H}))\circ\delta_{H}\;\footnotesize\textnormal{(by (vi.5) of Definition \ref{BTdef})}\\=& \mu_{H}\circ(H\otimes\lambda_{H}^{-1})\circ c_{H,H}^{-1}\circ\delta_{H}\circ T_{H}\;\footnotesize\textnormal{(by (vi.1) of Definition \ref{BTdef} and the condition of isomorphism for $c_{H,H}$)}\\=& \eta_{H}\circ\varepsilon_{H}\circ T_{H}\;\footnotesize\textnormal{(by \eqref{antipode} for $H^{cop}$)}\\=& \varepsilon_{H}\otimes\eta_{H}\;\footnotesize\textnormal{(by (vi.2) of Definition \ref{BTdef})}
		\end{align*}
Therefore, $H_{\sf{BT}}$ is a Hopf algebra in $\sf{C}$.

	To conclude the proof we have to show that (iii) of Definition \ref{H-brace} holds. Note that 
	\begin{equation}\label{Gbt=gamma}
	\Gamma_{H}^{\scalebox{0.6}{\sf{BT}}}=\gamma_{H}
	\end{equation}
	holds. Indeed,
		\begin{align*}
			&\Gamma_{H}^{\scalebox{0.6}{\sf{BT}}}\\=&\mu_{H}\circ(\lambda_{H}\otimes \mu_{H}^{\scalebox{0.6}{\sf{BT}}})\circ(\delta_{H}\otimes H)\;\footnotesize\textnormal{(by definition of $\Gamma_{H}^{\scalebox{0.6}{\sf{BT}}}$)}\\=& \mu_{H}\circ(\lambda_{H}\otimes(\mu_{H}\circ(H\otimes\gamma_{H})\circ(\delta_{H}\otimes H)))\circ(\delta_{H}\otimes H)\;\footnotesize\textnormal{(by definition of $\mu_{H}^{\scalebox{0.6}{\sf{BT}}}$)}\\=& \mu_{H}\circ((\lambda_{H}\ast id_{H})\otimes\gamma_{H})\circ(\delta_{H}\otimes H)\;\footnotesize\textnormal{(by associativity of $\mu_{H}$ and coassociativity of $\delta_{H}$)}\\=& \gamma_{H}\;\footnotesize\textnormal{(by \eqref{antipode} and (co)unit property)}.
		\end{align*}
		Consequently,
		\begin{align*}
			&\mu_{H}\circ(\mu_{H}^{\scalebox{0.6}{\sf{BT}}}\otimes\Gamma_{H}^{\scalebox{0.6}{\sf{BT}}})\circ(H\otimes c_{H,H}\otimes H)\circ(\delta_{H}\otimes H\otimes H)\\=& \mu_{H}\circ((\mu_{H}\circ(H\otimes\gamma_{H})\circ(\delta_{H}\otimes H))\otimes\gamma_{H})\circ(H\otimes c_{H,H}\otimes H)\circ(\delta_{H}\otimes H\otimes H)\\&\footnotesize\textnormal{(by definition of $\mu_{H}^{\scalebox{0.6}{\sf{BT}}}$ and \eqref{Gbt=gamma})}\\=& \mu_{H}\circ(H\otimes(\mu_{H}\circ(\gamma_{H}\otimes\gamma_{H})\circ(H\otimes c_{H,H}\otimes H)\circ(\delta_{H}\otimes H\otimes H)))\circ(\delta_{H}\otimes H\otimes H)\\&\footnotesize\textnormal{(by associativity of $\mu_{H}$ and coassociativity of $\delta_{H}$)}\\=&  \mu_{H}\circ(H\otimes\gamma_{H})\circ(\delta_{H}\otimes\mu_{H})\;\footnotesize\textnormal{(by (iii) of Definition \ref{BTdef})}\\=& \mu_{H}^{\scalebox{0.6}{\sf{BT}}}\circ(H\otimes\mu_{H})\;\footnotesize\textnormal{(by definition of $\mu_{H}^{\scalebox{0.6}{\sf{BT}}}$)}.
		\end{align*}
		
	Finally, by \eqref{Gbt=gamma} and thanks to axioms (i), (vi), (vi.3) and (vi.5) of Definition \ref{BTdef}, conditions (i), (ii), (ii.1) and (ii.2) of Definition \ref{sHbr} are obvious.
	\end{proof}
\begin{remark}
When $H$ is a cocommutative Hopf algebra, we recover \cite[Remark 4.5]{GGV}.
\end{remark}
\begin{corollary}
Let $(H,\gamma_{H},T_{H})$ be a brace triple in $\sf{C}$. Then, $(H,\gamma_{H})$ is a left $H_{\sf{BT}}$-module algebra.
\end{corollary}
\begin{proof}
Thanks to the fact that $\mathbb{H}_{\sf{BT}}=(H,H_{\sf{BT}})$ is a Hopf brace, we know that $(H,\Gamma_{H}^{\scalebox{0.6}{\sf{BT}}})$ is a left $H_{\sf{BT}}$-module algebra. Due to being $\Gamma_{H}^{\scalebox{0.6}{\sf{BT}}}=\gamma_{H}$, as we have proved in the previous result, we conclude that $(H,\gamma_{H})$ is a left $H_{\sf{BT}}$-module algebra.
\end{proof}
\begin{remark} Let's assume that $\sf{C}$ is symmetric. Under this assumption and thanks to the previous corollary, axiom (i) of Definition \ref{BTdef} means that $(H,\gamma_{H})$ is in the cocommutativity class of $H_{\sf{BT}}$.
\end{remark}
\begin{corollary}
Let $(H,\gamma_{H},T_{H})$ be a cocommutative brace triple, then 
\begin{equation}\label{TH2=ID}
T_{H}\circ T_{H}=id_{H}.
\end{equation} 
Therefore, conditions \emph{(vi.4)} and \emph{(vi.5)} of Definition \ref{BTdef} are equivalent in the cocommutative setting.
\end{corollary}
\begin{proof}
As was proved in Theorem \ref{Prop Brace TB}, $T_{H}$ is the antipode for the Hopf algebra $H_{\sf{BT}}$. Then, if $H$ is cocommutative, $H_{\sf{BT}}$ is cocommutative too and, as a consequence, \eqref{TH2=ID} holds.
\end{proof}
\begin{corollary}\label{compat ant}
If $f\colon (H,\gamma_{H}, T_{H})\rightarrow (B,\gamma_{B},T_{B})$ is a morphism of brace triples in $\sf{C}$, then $$f\circ T_{H}=T_{B}\circ f.$$
\end{corollary}
\begin{proof} It is enough to see that $f\colon H_{\sf{BT}}\rightarrow B_{\sf{BT}}$ is a Hopf algebra morphism. Due to the fact that $H_{\sf{BT}}$ and $B_{\sf{BT}}$ are Hopf algebras in $\sf{C}$ with the same underlying coalgebra structure and the same unit morphisms as $H$ and $B$, respectively, it is enough to prove that $f$ is compatible with the products $\mu_{H}^{\scalebox{0.6}{\sf{BT}}}$ and $\mu_{B}^{\scalebox{0.6}{\sf{BT}}}$. Indeed,
\begin{align*}
	&f\circ\mu_{H}^{\scalebox{0.6}{\sf{BT}}}\\=& f\circ\mu_{H}\circ(H\otimes\gamma_{H})\circ(\delta_{H}\otimes H)\;\footnotesize\textnormal{(by definition of $\mu_{H}^{\scalebox{0.6}{\sf{BT}}}$)}\\=& \mu_{B}\circ(f\otimes f)\circ(H\otimes\gamma_{H})\circ(\delta_{H}\otimes H)\;\footnotesize\textnormal{(by the condition of algebra morphism for $f\colon H\rightarrow B$)}\\=& \mu_{B}\circ(B\otimes\gamma_{B})\circ(((f\otimes f)\circ\delta_{H})\otimes f)\;\footnotesize\textnormal{(by \eqref{CondMorBT})}\\=& \mu_{B}\circ(B\otimes\gamma_{B})\circ(\delta_{B}\otimes B)\circ (f\otimes f)\;\footnotesize\textnormal{(by the condition of coalgebra morphism for $f$)}\\=& \mu_{B}^{\scalebox{0.6}{\sf{BT}}}\circ(f\otimes f)\;\footnotesize\textnormal{(by definition of $\mu_{B}^{\scalebox{0.6}{\sf{BT}}}$)}.
\end{align*}
So, due to being $f\colon H_{\sf{BT}}\rightarrow B_{\sf{BT}}$ a Hopf algebra morphism in $\sf{C}$, we can apply \eqref{morant} what concludes the proof.
\end{proof}
	
Theorem \ref{Prop Brace TB} implies that there exist a functor $F\colon \sf{BT}\longrightarrow\textnormal{\texttt{s}}\sf{HBr}$ defined by $F((H,\gamma_{H},T_{H}))=\mathbb{H}_{\sf{BT}}$ on objects and on morphisms by the identity. To see that $F$ is well-defined on morphisms, we have to prove that if $f$ is a morphism in $\sf{BT}$, then $f$ is a morphism in $\sf{HBr}$. To verify this fact, it is enough to compute that $f\circ \mu_{H}^{\scalebox{0.6}{\sf{BT}}}=\mu_{B}^{\scalebox{0.6}{\sf{BT}}}\circ(f\otimes f)$, what we have just seen in the proof of Corollary \ref{compat ant}. 

Moreover, we can also construct a brace triple from every \texttt{s}-Hopf brace. First of all, we are going to prove the following lemma.
\begin{lemma} Let $\mathbb{H}$ be a Hopf brace in $\sf{C}$. If $\Gamma_{H_{1}}$ satisfies condition \textnormal{(i)} of Definition \ref{sHbr}, then $\Gamma_{H_{1}}$ is a coalgebra morphism.
\end{lemma}
\begin{proof} On the one hand, it is straightforward to see that $\varepsilon_{H}\circ\Gamma_{H_{1}}=\varepsilon_{H}\otimes\varepsilon_{H}$. Let's see that $\delta_{H}\circ\Gamma_{H_{1}}=(\Gamma_{H_{1}}\otimes\Gamma_{H_{1}})\circ(H\otimes c_{H,H}\otimes H)\circ(\delta_{H}\otimes\delta_{H})$. Indeed:
\begin{align*}
&\delta_{H}\circ\Gamma_{H_{1}}\\=& \delta_{H}\circ\mu_{H}^{1}\circ(\lambda_{H}^{1}\otimes\mu_{H}^{2})\circ(\delta_{H}\otimes H)\;\footnotesize\textnormal{(by definition of $\Gamma_{H_{1}}$)}\\=& (\mu_{H}^{1}\otimes\mu_{H}^{1})\circ(H\otimes c_{H,H}\otimes H)\circ((\delta_{H}\circ\lambda_{H}^{1})\otimes(\delta_{H}\circ\mu_{H}^{2}))\circ(\delta_{H}\otimes H)\;\footnotesize\textnormal{(by the condition of coalgebra}\\&\footnotesize\textnormal{morphism for $\mu_{H}^{1}$)}\\=& (\mu_{H}^{1}\otimes\mu_{H}^{1})\circ(H\otimes c_{H,H}\otimes H)\circ((c_{H,H}\circ(\lambda_{H}^{1}\otimes\lambda_{H}^{1})\circ\delta_{H})\otimes((\mu_{H}^{2}\otimes\mu_{H}^{2})\circ(H\otimes c_{H,H}\otimes H)\\&\circ(\delta_{H}\otimes\delta_{H})))\circ(\delta_{H}\otimes H)\;\footnotesize\textnormal{(by the condition of coalgebra morphism for $\mu_{H}^{2}$ and \eqref{a-antip})}\\=& (H\otimes (\mu_{H}^{1}\circ(\lambda_{H}^{1}\otimes H)))\circ(((\Gamma_{H_{1}}\otimes H)\circ(H\otimes c_{H,H})\circ((c_{H,H}\circ\delta_{H})\otimes H))\otimes\mu_{H}^{2})\\&\circ(H\otimes c_{H,H}\otimes H)\circ(\delta_{H}\otimes\delta_{H})\;\footnotesize\textnormal{(by coassociativity of $\delta_{H}$, naturality of $c$ and definition of $\Gamma_{H_{1}}$)}\\=& (H\otimes (\mu_{H}^{1}\circ(\lambda_{H}^{1}\otimes H)))\circ(((\Gamma_{H_{1}}\otimes H)\circ(H\otimes c_{H,H})\circ(\delta_{H}\otimes H))\otimes\mu_{H}^{2})\circ(H\otimes c_{H,H}\otimes H)\\&\circ(\delta_{H}\otimes\delta_{H})\;\footnotesize\textnormal{(by (i) of Definition \ref{sHbr})}\\=& (\Gamma_{H_{1}}\otimes\Gamma_{H_{1}})\circ(H\otimes c_{H,H}\otimes H)\circ(\delta_{H}\otimes\delta_{H})\;\footnotesize\textnormal{(by coassociativity of $\delta_{H}$ and naturality of $c$)}.\qedhere
\end{align*}

\end{proof}
\begin{theorem}\label{sHB to BT} If $\mathbb{H}$ is an object in \textnormal{\texttt{s}}$\sf{HBr}$, then $(H_{1},\Gamma_{H_{1}},\lambda_{H}^{2})$ is a brace triple.
\end{theorem}
\begin{proof}It is a consequence of the following facts: By (i) of Definition \ref{sHbr} and previous lemma, conditions (i) and (ii) of Definition \ref{BTdef} hold. Moreover, it is well known that $(H_{1},\Gamma_{H_{1}})$ is a left $H_{2}$-module algebra. This property together with \eqref{eb2} implies that axioms (iii), (iv) and (v) of Definition \ref{BTdef} also hold. Identities (vi), (vi.1) and (vi.2) of Definition \ref{BTdef} follow by (ii) of Definition \ref{sHbr} and equations \eqref{a-antip} and \eqref{u-antip}. The remaining axioms, (vi.3), (vi.4) and (vi.5) of Definition \ref{BTdef}, are consequence of (ii.1) of Definition \ref{sHbr}, equation \eqref{GammaL2} and (ii.2) of Definition \ref{sHbr}, respectively.
\end{proof}
As a consequence of the previous theorem, we obtain a functor $G\colon \textnormal{\texttt{s}}\sf{HBr}\longrightarrow\sf{BT}$ acting on objects by $G(\mathbb{H})=(H_{1},\Gamma_{H_{1}},\lambda_{H}^{2})$ and on morphisms by the identity. To see that $G_{1}$ is well-defined on morphisms, we have to compute that if $f\colon\mathbb{H}\rightarrow\mathbb{B}$ is a morphism in $\textnormal{\texttt{s}}\sf{HBr}$, then $f\circ\Gamma_{H_{1}}=\Gamma_{B_{1}}\circ(f\otimes f)$. Indeed:
\begin{align*}
	&f\circ\Gamma_{H_{1}}\\=& f\circ\mu_{H}^{1}\circ(\lambda_{H}^{1}\otimes \mu_{H}^{2})\circ(\delta_{H}\otimes H)\;\footnotesize\textnormal{(by definition of $\Gamma_{H_{1}}$)}\\=&  \mu_{B}^{1}\circ((f\circ\lambda_{H}^{1})\otimes(f\circ\mu_{H}^{2}))\circ(\delta_{H}\otimes H)\;\footnotesize\textnormal{(by the condition of algebra morphism for $f\colon H_{1}\rightarrow B_{1}$)}\\=& \mu_{B}^{1}\circ(\lambda_{B}^{1}\otimes\mu_{B}^{2})\circ(((f\otimes f)\circ\delta_{H})\otimes f)\;\footnotesize\textnormal{(by the condition of algebra morphism for $f\colon H_{2}\rightarrow B_{2}$ and \eqref{morant})}\\=&  \mu_{B}^{1}\circ(\lambda_{B}^{1}\otimes\mu_{B}^{2})\circ(\delta_{B}\otimes B)\circ(f\otimes f)\;\footnotesize\textnormal{(by the condition of coalgebra morphism for $f$)}\\=& \Gamma_{B_{1}}\circ (f\otimes f)\;\footnotesize\textnormal{(by definition of $\Gamma_{B_{1}}$)}.
\end{align*}

Next theorem is the main result of this section. We will prove that functors $F$ and $G$ induce a categorical isomorphism between $\textnormal{\texttt{s}}\sf{HBr}$ and $\sf{BT}$.
\begin{theorem}\label{iso1}
The categories $\textnormal{\texttt{s}}\sf{HBr}$ and $\sf{BT}$ are isomorphic.
\end{theorem}
\begin{proof} 
First of all, it results clear that $G\circ F=\sf{id}_{\sf{BT}}$. Indeed, consider $(H,\gamma_{H},T_{H})$ a brace triple, we obtain that:
\begin{align*}
&(G\circ F)((H,\gamma_{H},T_{H}))\\=&G(\mathbb{H}_{\sf{BT}})\;\footnotesize\textnormal{(by definition of functor $F$)}\\=&(H,\Gamma_{H}^{\scalebox{0.6}{\sf{BT}}},T_{H})\;\footnotesize\textnormal{(by definition of functor $G$)}\\=&(H,\gamma_{H},T_{H})\;\footnotesize\textnormal{(by \eqref{Gbt=gamma})}.
\end{align*}

On the other side, consider $\mathbb{H}$ an object in $\textnormal{\texttt{s}}\sf{HBr}$. We have that:
\begin{align*}
&(F\circ G)(\mathbb{H})\\=&F((H_{1},\Gamma_{H_{1}},\lambda_{H}^{2}))\;\footnotesize\textnormal{(by definition of functor $G$)}\\=&(H_{1},H_{\sf{BT}})\;\footnotesize\textnormal{(by definition of functor $F$)},
\end{align*}
where, in this particular case,
\begin{align*}
&\mu_{H}^{\scalebox{0.6}{\sf{BT}}}\\=&\mu_{H}^{1}\circ(H\otimes\Gamma_{H_{1}})\circ(\delta_{H}\otimes H)\;\footnotesize\textnormal{(by definition of $\mu_{H}^{\scalebox{0.6}{\sf{BT}}}$)}\\=&\mu_{H}^{2}\;\footnotesize\textnormal{(by \eqref{eb2})}.
\end{align*}
Therefore, $H_{\sf{BT}}=H_{2}$, and then $F\circ G=\sf{id}_{\textnormal{\texttt{s}}\sf{HBr}}$.
\end{proof}

\begin{corollary}\label{Cor iso cocHB cocTB}
Categories $\sf{cocHBr}$ and $\sf{cocBT}$ are isomorphic.
\end{corollary}
\begin{proof}
It is enough to take into account Remarks \ref{TB-cocTB} and \ref{sHB-cocHB} and the previous theorem. The isomorphism in this case is given by functors $F'$ and $G'$ which are the restrictions of $F$ and $G$ to $\sf{cocBT}$ and $\sf{cocHBr}$, respectively. 
\vspace{-0.15cm}\[\xymatrix{
&\sf{BT}\ar@<1ex>[rr]^-{F} &\simeq &\sf{HBr}\ar@<1ex>[ll]^-{G}\\
&\sf{cocBT}\ar@<1ex>[rr]^-{F'}\ar@{^(->}[u] &\simeq &\sf{cocHBr}\ar@<1ex>[ll]^-{G'}\ar@{^(->}[u]
}\]
\end{proof}
\section{Post-Hopf algebras and Hopf braces}\label{sec3}
In this section we introduce the notion of post-Hopf algebra in the braided monoidal context. In particular, for the category of vector spaces over a field $\mathbb{K}$, we obtain the concept of post-Hopf algebra presented in \cite{LST}, where the authors get an equivalence between Hopf braces and these objects under cocommutativity assumption. Besides being working in a more general setting, in this section we prove that the categories of finite cocommutative Hopf braces and cocommutative post-Hopf algebras satisfying condition \eqref{property for betatilde2} are isomoprhic. As a consequence of this result together with Corollary \ref{Cor iso cocHB cocTB}, we also deduce that finite cocommutative brace triples are isomorphic to post-Hopf algebras verifying \eqref{property for betatilde2}.

\begin{definition}\label{PHopf algebra}
A post-Hopf algebra in $\sf{C}$ is a pair $(H,m_{H})$ where $H$ is a finite Hopf algebra in $\sf{C}$ and $m_{H}\colon H\otimes H\rightarrow H$ is a morphism in $\sf{C}$ that satisfies the following conditions:
\begin{itemize}
	\item[(i)] $m_{H}$ is a coalgebra morphism, which means that the following equalities are satisfied:
	\begin{itemize}
	\item[(i.1)] $\delta_{H}\circ m_{H}=(m_{H}\otimes m_{H})\circ(H\otimes c_{H,H}\otimes H)\circ(\delta_{H}\otimes\delta_{H})$,
	\item[(i.2)] $\varepsilon_{H}\circ m_{H}=\varepsilon_{H}\otimes\varepsilon_{H}.$
	\end{itemize}
	\item[(ii)] $m_{H}\circ(H\otimes m_{H})=m_{H}\circ((\mu_{H}\circ(H\otimes m_{H})\circ(\delta_{H}\otimes H))\otimes H)$, which is called the ``weighted'' associativity.
	\item[(iii)] $m_{H}\circ(H\otimes\mu_{H})=\mu_{H}\circ(m_{H}\otimes m_{H})\circ(H\otimes c_{H,H}\otimes H)\circ(\delta_{H}\otimes H\otimes H)$.
	\item[(iv)] The morphism 
	\[\alpha_{H}\coloneqq (H^{\ast}\otimes m_{H})\circ(c_{H,H^{\ast}}\otimes H)\circ(H\otimes   a_{H}(K))\colon H\rightarrow H^{\ast}\otimes H\]
	is convolution invertible in $\Hom(H,H^{\ast}\otimes H)$, which means that there exists $\beta_{H}\colon H\rightarrow H^{\ast}\otimes H$ such that
	\[(H^{\ast}\otimes  b_{H}(K)\otimes H)\circ(\alpha_{H}\otimes\beta_{H})\circ\delta_{H}=\varepsilon_{H}\otimes  a_{H}(K)=(H^{\ast}\otimes  b_{H}(K)\otimes H)\circ (\beta_{H}\otimes\alpha_{H})\circ\delta_{H}.\]
	\end{itemize}
\end{definition}
\begin{remark}
Given a post-Hopf algebra $(H,m_{H})$, conditions (i) and (iii) of Definition \ref{PHopf algebra} imply that 
\begin{equation}\label{oldeq3}
m_{H}\circ(H\otimes\eta_{H})=\varepsilon_{H}\otimes \eta_{H}
\end{equation}
holds by Theorem \ref{th.interest}.
\end{remark}
\begin{definition}
Let $(H,m_{H})$ and $(B,m_{B})$ be post-Hopf algebras in $\sf{C}$ and let $f\colon H\rightarrow B$ be a morphism in $\sf{C}$. We will say that $f$ is a post-Hopf algebra morphism if $f$ is a Hopf algebra morphism and the condition
\begin{equation}\label{CondMorPHopf}
f\circ m_{H}=m_{B}\circ(f\otimes f)
\end{equation}
holds.
\end{definition}
Post-Hopf algebras and their morphisms form a category and we will denote it by $\sf{Post}$-${\sf Hopf}$. When $H$ is cocommutative, we will say that $(H,m_{H})$ is a cocommutative post-Hopf algebra in $\sf{C}$. Cocommutative post-Hopf algebras constitute a full subcategory of $\sf{Post}$-${\sf Hopf}$ which we will denote by $\sf{cocPost}$-$\sf{Hopf}$.
\begin{lemma}\label{Propiedades1PHA}
	Let $(H,m_{H})$ be an object in $\sf{Post}$-${\sf Hopf}$. It is verified that 
\begin{equation}\label{eq1}
	m_{H}\circ c_{H,H}^{-1}=( b_{H}(K)\otimes H)\circ(H\otimes \alpha_{H}).
\end{equation}
Therefore, 
\begin{equation}\label{eq2}
	m_{H}=( b_{H}(K)\otimes H)\circ(H\otimes \alpha_{H})\circ c_{H,H}.
\end{equation}
\end{lemma}
\begin{proof}
Let's start proving \eqref{eq1}:
	\begin{align*}
		&( b_{H}(K)\otimes H)\circ(H\otimes\alpha_{H})\\=& ( b_{H}(K)\otimes m_{H})\circ(H\otimes ((c_{H,H^{\ast}}\otimes H)\circ( H\otimes  a_{H}(K))))\;\footnotesize\textnormal{(by definition of $\alpha_{H}$)}\\=&(b_{H}(K)\otimes (m_{H}\circ c_{H,H}^{-1}))\circ(H\otimes a_{H}(K)\otimes H)\;\footnotesize\textnormal{(by \eqref{finitebraid})}\\=&m_{H}\circ c_{H,H}^{-1}\;\footnotesize\textnormal{(by the adjunction properties)}.
	\end{align*}
So, composing on the right with $c_{H,H}$, we obtain \eqref{eq2}.
\end{proof}
\begin{lemma}\label{Propiedades2PHA}
	Let $(H, m_{H})$ be an object in $\sf{Post}$-${\sf Hopf}$, then
		\begin{equation}\label{eq4}
			m_{H}\circ(\eta_{H}\otimes H)=id_{H}.
		\end{equation}
\end{lemma}
\begin{proof}
First of all, note that the morphism $m_{H}\circ(\eta_{H}\otimes H)$ is idempotent. Indeed,
\begin{align*}
	 &m_{H} \circ(H\otimes m_{H} )\circ(\eta_{H}\otimes\eta_{H}\otimes H)\\\nonumber =&  m_{H} \circ((\mu_{H}\circ(H\otimes m_{H} )\circ(\delta_{H}\otimes H)\circ(\eta_{H}\otimes\eta_{H}))\otimes H)\;\footnotesize\textnormal{(by (ii) of Definition \ref{PHopf algebra})}\\\nonumber =& m_{H} \circ((\mu_{H}\circ(H\otimes m_{H} )\circ(\eta_{H}\otimes\eta_{H}\otimes\eta_{H}))\otimes H)\;\footnotesize\textnormal{(by the condition of coalgebra morphism for $\eta_{H}$)}\\\nonumber =& m_{H} \circ(( m_{H} \circ(\eta_{H}\otimes \eta_{H}))\otimes H)\;\footnotesize\textnormal{(by unit property)}\\\nonumber =&(\varepsilon_{H}\circ\eta_{H})\otimes( m_{H} \circ(\eta_{H}\otimes H))\;\footnotesize\textnormal{(by \eqref{oldeq3})}\\\nonumber =& m_{H} \circ(\eta_{H}\otimes H)\;\footnotesize\textnormal{(by (co)unit property)}.
\end{align*}
Therefore, the equality
\begin{equation}\label{equality2Lemma3.5}
( b_{H}(K)\otimes b_{H}(K)\otimes H)\circ(H\otimes((\alpha_{H}\otimes\alpha_{H})\circ(\eta_{H}\otimes\eta_{H})))=( b_{H}(K)\otimes H)\circ(H\otimes(\alpha_{H}\circ\eta_{H}))
\end{equation}
holds because
\begin{align*}
	&( b_{H}(K)\otimes b_{H}(K)\otimes H)\circ(H\otimes((\alpha_{H}\otimes\alpha_{H})\circ(\eta_{H}\otimes\eta_{H})))\\\nonumber =& m_{H} \circ(H\otimes m_{H} )\circ(\eta_{H}\otimes\eta_{H}\otimes H)\;\footnotesize\textnormal{(by \eqref{eq1} and naturality of $c$)}\\\nonumber =& m_{H} \circ(\eta_{H}\otimes H)\;\footnotesize\textnormal{(by the idempotent character of $m_{H}\circ(\eta_{H}\otimes H)$)}\\\nonumber =&( b_{H}(K)\otimes H)\circ(H\otimes(\alpha_{H}\circ\eta_{H}))\;\footnotesize\textnormal{(by \eqref{eq1} and naturality of $c$)}.
\end{align*}
Then, using the previous equalities and the finite character of $H$, we have that
\begin{align*}
	&(H^{\ast}\otimes b_{H}(K)\otimes H)\circ(\alpha_{H}\otimes\alpha_{H})\circ(\eta_{H}\otimes\eta_{H})\\\nonumber =&(H^{\ast}\otimes b_{H}(K)\otimes b_{H}(K)\otimes H)\circ(  a_{H}(K)\otimes(\alpha_{H}\circ\eta_{H})\otimes(\alpha_{H}\circ\eta_{H}))\;\footnotesize\textnormal{(by the adjunction properties)}\\\nonumber =&(H^{\ast}\otimes b_{H}(K)\otimes H)\circ(  a_{H}(K)\otimes(\alpha_{H}\circ\eta_{H}))\;\footnotesize\textnormal{(by \eqref{equality2Lemma3.5})}\\\nonumber =&\alpha_{H}\circ\eta_{H}\;\footnotesize\textnormal{(by the adjunction properties)}.
\end{align*}
As a consequence, if $\beta_{H}$ is the convolution inverse of $\alpha_{H}$ in $\Hom(H,H^{\ast}\otimes H)$, we deduce the following:
\begin{align*}
	&(H^{\ast}\otimes b_{H}(K)\otimes b_{H}(K)\otimes H)\circ(\beta_{H}\otimes\alpha_{H}\otimes\alpha_{H})\circ((\delta_{H}\circ\eta_{H})\otimes\eta_{H})\\=& (H^{\ast}\otimes b_{H}(K)\otimes b_{H}(K)\otimes H)\circ(\beta_{H}\otimes\alpha_{H}\otimes\alpha_{H})\circ(\eta_{H}\otimes\eta_{H}\otimes\eta_{H})\;\footnotesize\textnormal{(by the condition of coalgebra}\\&\footnotesize\textnormal{morphism for $\eta_{H}$)}\\=& (H^{\ast}\otimes b_{H}(K)\otimes H)\circ(\beta_{H}\otimes\alpha_{H})\circ(\eta_{H}\otimes\eta_{H})\;\footnotesize\textnormal{(by \eqref{equality2Lemma3.5})}\\=& (H^{\ast}\otimes b_{H}(K)\otimes H)\circ(\beta_{H}\otimes\alpha_{H})\circ\delta_{H}\circ\eta_{H}\;\footnotesize\textnormal{(by the condition of coalgebra morphism for $\eta_{H}$)}\\=& (\varepsilon_{H}\circ\eta_{H})\otimes  a_{H}(K)\;\footnotesize\textnormal{(by (iv) of Definition \ref{PHopf algebra})}\\=&   a_{H}(K)\;\footnotesize\textnormal{(by the (co)unit properties)}
\end{align*}
and, on the other hand:
\begin{align*}\label{proof lema5}
	&(H^{\ast}\otimes b_{H}(K)\otimes b_{H}(K)\otimes H)\circ(\beta_{H}\otimes\alpha_{H}\otimes\alpha_{H})\circ((\delta_{H}\circ\eta_{H})\otimes\eta_{H})\\=&(\varepsilon_{H}\circ\eta_{H})\otimes((H^{\ast}\otimes b_{H}(K)\otimes H)\circ(  a_{H}(K)\otimes(\alpha_{H}\circ\eta_{H})))\;\footnotesize\textnormal{(by (iv) of Definition \ref{PHopf algebra})}\\=&\alpha_{H}\circ\eta_{H}\;\footnotesize\textnormal{(by (co)unit properties and the adjunction properties)}.
\end{align*}
So, by the two previous equalities, we obtain that \begin{equation}\label{proofequal}\alpha_{H}\circ\eta_{H}=a_{H}(K).\end{equation} Therefore, we conclude the proof as follows:
\begin{align*}
&id_{H}\\=&( b_{H}(K)\otimes H)\circ(H\otimes  a_{H}(K))\;\footnotesize\textnormal{(by the adjunction properties)}\\=&( b_{H}(K)\otimes H)\circ(H\otimes  (\alpha_{H}\circ\eta_{H}))\;\footnotesize\textnormal{(by \eqref{proofequal})}\\=&( b_{H}(K)\otimes m_{H} )\circ(H\otimes ((c_{H,H^{\ast}}\otimes H)\circ (\eta_{H}\otimes  a_{H}(K))))\;\footnotesize\textnormal{(by definition of $\alpha_{H}$)}\\=&(b_{H}(K)\otimes(m_{H}\circ c_{H,H}^{-1}))\circ(H\otimes a_{H}(K)\otimes \eta_{H})\;\footnotesize\textnormal{(by \eqref{finitebraid})}\\=&m_{H}\circ(\eta_{H}\otimes H)\;\footnotesize\textnormal{(by naturality of $c$ and the adjunction properties)}.\qedhere
\end{align*}
\end{proof}

The goal of the following results will consist on building a post-Hopf algebra from a brace triple. Suppose that $(H,\gamma_{H},T_{H})$ is a brace triple in $\sf{C}$ with $H$ finite. Note that conditions (ii), (iii) and (iv) of Definition \ref{BTdef} imply that $\gamma_{H}$ satisfies (i), (iii) and (ii) of Definition \ref{PHopf algebra}, respectively. So, in order to construct a post-Hopf algebra from a brace triple, it is enough to prove that $\alpha_{H}=(H^{\ast}\otimes\gamma_{H})\circ(c_{H,H^{\ast}}\otimes H)\circ(H\otimes  a_{H}(K))$ is convolution invertible in $\Hom(H,H^{\ast}\otimes H)$.

\begin{theorem}\label{inversebetaBT}
	Let $(H,\gamma_{H},T_{H})$ be a brace triple in $\sf{C}$ with $H$ finite. The morphism $\alpha_{H}=(H^{\ast}\otimes\gamma_{H})\circ(c_{H,H^{\ast}}\otimes H)\circ(H\otimes  a_{H}(K))$ is invertible for the convolution in $\Hom(H,H^{\ast}\otimes H)$ with inverse
	\[\beta_{H}\coloneqq \alpha_{H}\circ T_{H}^{-1}.\]
\end{theorem}
\begin{proof}
	On the one side,
	\begin{align*}
		&(H^{\ast}\otimes b_{H}(K)\otimes H)\circ(\beta_{H}\otimes\alpha_{H})\circ\delta_{H}\\=& (H^{\ast}\otimes b_{H}(K)\otimes H)\circ(((H^{\ast}\otimes\gamma_{H})\circ(c_{H,H^{\ast}}\otimes H)\circ(T_{H}^{-1}\otimes  a_{H}(K)))\otimes((H^{\ast}\otimes\gamma_{H})\circ(c_{H,H^{\ast}}\otimes H)\\&\circ(H\otimes  a_{H}(K))))\circ\delta_{H}\;\footnotesize\textnormal{(by definition of $\alpha_{H}$ and $\beta_{H}$)}\\=& (H^{\ast}\otimes(\gamma_{H}\circ c_{H,H}^{-1}\circ(\gamma_{H}\otimes H)))\circ(((c_{H,H^{\ast}}\otimes H)\circ(T_{H}^{-1}\otimes  a_{H}(K)))\otimes H)\circ\delta_{H}\;\footnotesize\textnormal{(by \eqref{finitebraid} and the}\\&\footnotesize\textnormal{adjunction properties)}\\=& (H^{\ast}\otimes(\gamma_{H}\circ(H\otimes\gamma_{H})))\circ(((H^{\ast}\otimes c_{H,H}^{-1})\circ(c_{H,H^{\ast}}\otimes H)\circ(T_{H}^{-1}\otimes c_{H,H^{\ast}}))\otimes H)\circ(\delta_{H}\otimes  a_{H}(K))\\&\footnotesize\textnormal{(by naturality of $c$)}\\=& (H^{\ast}\otimes\gamma_{H})\circ(((H^{\ast}\otimes\mu_{H}^{\scalebox{0.6}{\sf{BT}}})\circ(c_{H,H^{\ast}}\otimes H)\circ(H\otimes c_{H,H^{\ast}}))\otimes H)\circ((c_{H,H}^{-1}\circ(T_{H}^{-1}\otimes H)\circ\delta_{H})\otimes  a_{H}(K))\\&\footnotesize\textnormal{(by naturality of $c$ and (iv) of Definition \ref{BTdef})}\\=& (H^{\ast}\otimes\gamma_{H})\circ(c_{H,H^{\ast}}\otimes H)\circ((\mu_{H}^{\scalebox{0.6}{\sf{BT}}}\circ c_{H,H}^{-1}\circ(T_{H}^{-1}\otimes H)\circ\delta_{H})\otimes  a_{H}(K))\;\footnotesize\textnormal{(by naturality of $c$)}\\=& (H^{\ast}\otimes\gamma_{H})\circ(c_{H,H^{\ast}}\otimes H)\circ((\eta_{H}\circ\varepsilon_{H})\otimes  a_{H}(K))\;\footnotesize\textnormal{(by \eqref{antipode} for $H_{\sf{BT}}^{cop}$)}\\=& \varepsilon_{H}\otimes  a_{H}(K)\;\footnotesize\textnormal{(by naturality of $c$ and (v) of Definition \ref{BTdef})}.
	\end{align*}
	On the other side,
	\begin{align*}
		&(H^{\ast}\otimes b_{H}(K)\otimes H)\circ(\alpha_{H}\otimes\beta_{H})\circ\delta_{H}\\=& (H^{\ast}\otimes b_{H}(K)\otimes H)\circ(((H^{\ast}\otimes\gamma_{H})\circ(c_{H,H^{\ast}}\otimes H)\circ(H\otimes  a_{H}(K)))\otimes((H^{\ast}\otimes\gamma_{H})\circ(c_{H,H^{\ast}}\otimes H)\\&\circ(T_{H}^{-1}\otimes  a_{H}(K))))\circ\delta_{H}\;\footnotesize\textnormal{(by definition of $\alpha_{H}$ and $\beta_{H}$)}\\=& (H^{\ast}\otimes(\gamma_{H}\circ c_{H,H}^{-1}))\circ(((H^{\ast}\otimes\gamma_{H})\circ(c_{H,H^{\ast}}\otimes H)\circ(H\otimes  a_{H}(K)))\otimes T_{H}^{-1})\circ\delta_{H}\;\footnotesize\textnormal{(by \eqref{finitebraid} and the}\\&\footnotesize\textnormal{adjunction properties)}\\=& (H^{\ast}\otimes(\gamma_{H}\circ(H\otimes\gamma_{H})\circ(c_{H,H}^{-1}\otimes H)))\circ (((c_{H,H^{\ast}}\otimes H)\circ (H\otimes c_{H,H^{\ast}}))\otimes H)\\&\circ(((H\otimes T_{H}^{-1})\circ\delta_{H})\otimes  a_{H}(K))\;\footnotesize\textnormal{(by naturality of $c$)}\\=& (H^{\ast}\otimes\gamma_{H})\circ(c_{H,H^{\ast}}\otimes H)\circ((\mu_{H}^{\scalebox{0.6}{\sf{BT}}}\circ(T_{H}^{-1}\otimes H)\circ c_{H,H}^{-1}\circ\delta_{H})\otimes  a_{H}(K))\;\footnotesize\textnormal{(by naturality of $c$ and (iv) of}\\&\footnotesize\textnormal{Definition \ref{BTdef})}\\=& (H^{\ast}\otimes\gamma_{H})\circ(c_{H,H^{\ast}}\otimes H)\circ((\eta_{H}\circ\varepsilon_{H})\otimes  a_{H}(K))\;\footnotesize\textnormal{(by \eqref{antipode} for $H_{\sf{BT}}^{cop}$)}\\=& \varepsilon_{H}\otimes   a_{H}(K)\;\footnotesize\textnormal{(by naturality of $c$ and (v) of Definition \ref{BTdef})}.\qedhere
	\end{align*}
\end{proof}

Previous theorem can be interpreted in a functorial way as follows: If we denote by $\sf{BT^{f}}$ the subcategory of brace triples whose underlying Hopf algebra is finite, then there exists a functor
$P\colon \sf{BT^{f}}\longrightarrow\sf{Post}\textnormal{-}\sf{Hopf}$
acting on objects by $P((H,\gamma_{H},T_{H}))=(H,\gamma_{H})$ and on morphisms by the identity.

\begin{theorem}\label{Th bialg subadj} Let $(H,m_{H})$ be an object in $\sf{cocPost}$-$\sf{Hopf}$, then			\[\widehat{H}=(H,\eta_{H},\widehat{\mu}_{H},\varepsilon_{H},\delta_{H})\]
	is a bialgebra in $\sf{C}$, where $\widehat{\mu}_{H}\coloneqq \mu_{H}\circ(H\otimes m_{H} )\circ(\delta_{H}\otimes H).$
\end{theorem}
\begin{proof}  Note that we already know that $(H,\varepsilon_{H},\delta_{H})$ is a coalgebra in $\sf{C}$ and that $\eta_{H}$ is a coalgebra morphism. Then, firstly, we have to compute that $(H,\eta_{H},\widehat{\mu}_{H})$ is an algebra in $\sf{C}$. Indeed, let's start proving the unit property. On the one hand,
\begin{align*}
	&\widehat{\mu}_{H}\circ(\eta_{H}\otimes H)\\=& \mu_{H}\circ(H\otimes m_{H} )\circ((\delta_{H}\circ\eta_{H})\otimes H)\;\footnotesize\textnormal{(by definition of $\widehat{\mu}_{H}$)}\\=& \mu_{H}\circ(H\otimes m_{H} )\circ(\eta_{H}\otimes\eta_{H}\otimes H)\;\footnotesize\textnormal{(by the condition of coalgebra morphism for $\eta_{H}$)}\\=&   m_{H} \circ(\eta_{H}\otimes H)\;\footnotesize\textnormal{(by unit properties)}\\=&  id_{H}\;\footnotesize\textnormal{(by \eqref{eq4})}
	\end{align*}
	and, on the other hand,
	\begin{align*}
&\widehat{\mu}_{H}\circ(H\otimes \eta_{H})\\=& \mu_{H}\circ(H\otimes m_{H} )\circ(\delta_{H}\otimes \eta_{H})\;\footnotesize\textnormal{(by definition of $\widehat{\mu}_{H}$)}\\=& \mu_{H}\circ(H\otimes\varepsilon_{H}\otimes\eta_{H})\circ\delta_{H}\;\footnotesize\textnormal{(by \eqref{oldeq3})}\\=& id_{H}\;\footnotesize\textnormal{(by (co)unit properties)}.
\end{align*}
The associativity of $\widehat{\mu}_{H}$ follows by:
\begin{align*}
	&\widehat{\mu}_{H}\circ(\widehat{\mu}_{H}\otimes H)\\=& \mu_{H}\circ(H\otimes m_{H} )\circ((\delta_{H}\circ\mu_{H}\circ(H\otimes m_{H} )\circ(\delta_{H}\otimes H))\otimes H)\;\footnotesize\textnormal{(by definition of $\widehat{\mu}_{H}$)}\\=& \mu_{H}\circ(\mu_{H}\otimes( m_{H} \circ(\mu_{H}\otimes H)))\circ(H\otimes c_{H,H}\otimes H\otimes H)\circ (\delta_{H}\otimes(\delta_{H}\circ m_{H} )\otimes H)\circ(\delta_{H}\otimes H\otimes H)\\&\footnotesize\textnormal{(by the condition of coalgebra morphism for $\mu_{H}$)}\\=& \mu_{H}\circ(\mu_{H}\otimes( m_{H} \circ(\mu_{H}\otimes H)))\circ(H\otimes c_{H,H}\otimes H\otimes H)\circ (\delta_{H}\otimes(( m_{H} \otimes m_{H} )\circ(H\otimes c_{H,H}\otimes H)\\&\circ(\delta_{H}\otimes \delta_{H}))\otimes H)\circ(\delta_{H}\otimes H\otimes H)\;\footnotesize\textnormal{(by (i.1) of Definition \ref{PHopf algebra})}\\=& \mu_{H}\circ((\mu_{H}\circ(H\otimes m_{H}))\otimes(m_{H}\circ(\mu_{H}\otimes H)\circ(H\otimes m_{H}\otimes H)))\circ(((H\otimes H\otimes c_{H,H}\otimes H\otimes H)\\&\circ(H\otimes c_{H,H}\otimes c_{H,H}\otimes H)\circ(((\delta_{H}\otimes\delta_{H})\circ\delta_{H})\otimes\delta_{H}))\otimes H)\;\footnotesize\textnormal{(by naturality of $c$)}\\=& \mu_{H}\circ((\mu_{H}\circ(H\otimes m_{H})\circ(\delta_{H}\otimes H))\otimes(m_{H}\circ(\mu_{H}\otimes H)))\\&\circ(H\otimes ((c_{H,H}\otimes m_{H})\circ(H\otimes c_{H,H}\otimes H)\circ(\delta_{H}\otimes H\otimes H))\otimes H)\circ(\delta_{H}\otimes\delta_{H}\otimes H)\\&\footnotesize\textnormal{(by coassociativity and cocommutativity of $\delta_{H}$)}\\=& \mu_{H}\circ((\mu_{H}\circ(H\otimes m_{H})\circ(\delta_{H}\otimes H))\otimes(m_{H}\circ((\mu_{H}\circ(H\otimes m_{H})\circ(\delta_{H}\otimes H))\otimes H)))\\&\circ(((H\otimes c_{H,H}\otimes H)\circ(\delta_{H}\otimes\delta_{H}))\otimes H)\;\footnotesize\textnormal{(by naturality of $c$)}\\=& \mu_{H}\circ((\mu_{H}\circ(H\otimes m_{H})\circ(\delta_{H}\otimes H))\otimes(m_{H}\circ(H\otimes m_{H})))\circ(((H\otimes c_{H,H}\otimes H)\circ(\delta_{H}\otimes\delta_{H}))\otimes H)\\&\footnotesize\textnormal{(by (ii) of Definition \ref{PHopf algebra})}\\=& \mu_{H}\circ(H\otimes(\mu_{H}\circ(m_{H}\otimes m_{H})\circ(H\otimes c_{H,H}\otimes H)\circ(\delta_{H}\otimes H\otimes H)))\circ(H\otimes H\otimes H\otimes m_{H})\\&\circ(\delta_{H}\otimes\delta_{H}\otimes H)\;\footnotesize\textnormal{(by coassociativity of $\delta_{H}$ and associativity of $\mu_{H}$)}\\=& \mu_{H}\circ(H\otimes m_{H})\circ(\delta_{H}\otimes(\mu_{H}\circ(H\otimes m_{H})\circ(\delta_{H}\otimes H)))\;\footnotesize\textnormal{(by (iii) of Definition \ref{PHopf algebra})}\\=& \widehat{\mu}_{H}\circ(H\otimes\widehat{\mu}_{H})\;\footnotesize\textnormal{(by definition of $\widehat{\mu}_{H}$)}.
\end{align*}

Finally, we will prove that $\widehat{\mu}_{H}$ is a coalgebra morphism. By the condition of coalgebra morphism for $\mu_{H}$, (i.2) of Definition \ref{PHopf algebra} and the counit property, it is straightforward to compute that $\varepsilon_{H}\circ\widehat{\mu}_{H}=\varepsilon_{H}\otimes\varepsilon_{H}$. Moreover,
\begin{align*}
	&\delta_{H}\circ\widehat{\mu}_{H}\\=& \delta_{H}\circ\mu_{H}\circ(H\otimes  m_{H} )\circ(\delta_{H}\otimes H)\;\footnotesize\textnormal{(by definition of $\widehat{\mu}_{H}$)}\\=& (\mu_{H}\otimes\mu_{H})\circ(H\otimes c_{H,H}\otimes H)\circ(\delta_{H}\otimes(\delta_{H}\circ m_{H}))\circ(\delta_{H}\otimes H)\;\footnotesize\textnormal{(by the condition of coalgebra}\\&\footnotesize\textnormal{morphism for $\mu_{H}$)}\\=& (\mu_{H}\otimes\mu_{H})\circ(H\otimes c_{H,H}\otimes H)\circ(\delta_{H}\otimes (( m_{H} \otimes m_{H} )\circ(H\otimes c_{H,H}\otimes H)\circ(\delta_{H}\otimes\delta_{H})))\\&\circ(\delta_{H}\otimes H)\;\footnotesize\textnormal{(by (i.1) of Definition \ref{PHopf algebra})}\\=& ((\mu_{H}\circ(H\otimes m_{H}))\otimes(\mu_{H}\circ(H\otimes m_{H})))\circ(H\otimes((H\otimes c_{H,H}\otimes H)\circ(c_{H,H}\otimes c_{H,H}))\otimes H)\\&\circ(((\delta_{H}\otimes\delta_{H})\circ\delta_{H})\otimes\delta_{H})\;\footnotesize\textnormal{(by naturality of $c$)}\\=& ((\mu_{H}\circ(H\otimes m_{H})\circ(\delta_{H}\otimes H))\otimes(\mu_{H}\circ(H\otimes m_{H})))\circ(H\otimes((c_{H,H}\otimes H)\circ(H\otimes c_{H,H})\\&\circ(\delta_{H}\otimes H))\otimes H)\circ(\delta_{H}\otimes\delta_{H})\;\footnotesize\textnormal{(by cocommutativity and coassociativity of $\delta_{H}$)}\\=& ((\mu_{H}\circ(H\otimes m_{H} )\circ(\delta_{H}\otimes H))\otimes(\mu_{H}\circ(H\otimes m_{H} )\circ(\delta_{H}\otimes H)))\circ(H\otimes c_{H,H}\otimes H)\\&\circ(\delta_{H}\otimes \delta_{H})\;\footnotesize\textnormal{(by naturality of $c$)}\\=& (\widehat{\mu}_{H}\otimes\widehat{\mu}_{H})\circ(H\otimes c_{H,H}\otimes H)\circ(\delta_{H}\otimes\delta_{H})\;\footnotesize\textnormal{(by definition of $\widehat{\mu}_{H}$)}.\qedhere
\end{align*}
\end{proof}
\begin{corollary}
	If $(H,m_{H})$ is an object in $\sf{cocPost}$-$\sf{Hopf}$, then $(H,m_{H})$ is a left $\widehat{H}$-module algebra-coalgebra.
\end{corollary}
\begin{proof} It is a consequence of the following facts: Thanks to conditions (ii) of Definition \ref{PHopf algebra} and \eqref{eq4}, $(H,m_{H})$ is a left $\widehat{H}$-module. Moreover, by \eqref{oldeq3} and (iii) of Definition \ref{PHopf algebra}, $\eta_{H}$ and $\mu_{H}$ are morphisms of left $\widehat{H}$-modules, respectively. To finish, $m_{H}$ is a coalgebra morphism by (i) of Definition \ref{PHopf algebra} which implies that $(H,m_{H})$ is a left $\widehat{H}$-module coalgebra.
\end{proof}
Along the following results, we are going to study some properties about the morphism
\begin{align*}&\widehat{\lambda}_{H}\coloneqq ( b_{H}(K)\otimes H)\circ(c_{H^{\ast},H}\otimes H)\circ(H^{\ast}\otimes c_{H,H})\circ(\beta_{H}\otimes\lambda_{H})\circ\delta_{H}\\=& ( b_{H}(K)\otimes H)\circ(H\otimes\beta_{H})\circ c_{H,H}\circ(H\otimes\lambda_{H})\circ\delta_{H}\;\footnotesize\textnormal{(by naturality of $c$)}\end{align*}
with the final objective of proving that it is the antipode for $\widehat{H}$. We are going to denote by $\hat{\ast}$ to the convolution in $\mathcal{H}(H,\widehat{H})$.
\begin{remark}
	First of all, note that 
	\begin{equation}\label{lambdatr cocom}
	\widehat{\lambda}_{H}=( b_{H}(K)\otimes H)\circ(\lambda_{H}\otimes\beta_{H})\circ\delta_{H}
	\end{equation}
	when $(H,m_{H})$ is a cocommutative post-Hopf algebra. Indeed,
	\begin{align*}
		&\widehat{\lambda}_{H}= ( b_{H}(K)\otimes H)\circ(c_{H^{\ast},H}\otimes H)\circ(H^{\ast}\otimes c_{H,H})\circ(\beta_{H}\otimes\lambda_{H})\circ\delta_{H}\\\nonumber =&( b_{H}(K)\otimes H)\circ(\lambda_{H}\otimes\beta_{H})\circ c_{H,H}\circ\delta_{H}\;\footnotesize\textnormal{(by naturality of $c$)}\\\nonumber =&( b_{H}(K)\otimes H)\circ(\lambda_{H}\otimes\beta_{H})\circ \delta_{H}\;\footnotesize\textnormal{(by cocommutativity of $\delta_{H}$)}.
	\end{align*}
\end{remark}
\begin{lemma}\label{relation lambdaH y lambdaHhat}
	If $(H, m_{H} )$ is a cocommutative post-Hopf algebra in $\sf{C}$, then
	\begin{equation}\label{prop lambdatr 1}
		 m_{H} \circ(H\otimes\widehat{\lambda}_{H})\circ\delta_{H}=\lambda_{H}.
	\end{equation}
	As a consequence,
		\begin{equation}\label{convolution1}
		id_{H}\hatast\widehat{\lambda}_{H}=\varepsilon_{H}\otimes\eta_{H}.
	\end{equation}
\end{lemma}
\begin{proof} Let's start proving \eqref{prop lambdatr 1}:
	\begin{align*}
&m_{H} \circ(H\otimes\widehat{\lambda}_{H})\circ\delta_{H}\\=&m_{H}\circ(H\otimes((b_{H}(K)\otimes H)\circ(\lambda_{H}\otimes\beta_{H})\circ\delta_{H}))\circ\delta_{H}\;\footnotesize\textnormal{(by \eqref{lambdatr cocom})}\\=&(b_{H}(K)\otimes H)\circ(H\otimes\alpha_{H})\circ c_{H,H}\circ(H\otimes((b_{H}(K)\otimes H)\circ(\lambda_{H}\otimes\beta_{H})\circ\delta_{H}))\circ\delta_{H}\;\footnotesize\textnormal{(by \eqref{eq2})}\\=&((b_{H}(K)\circ(\lambda_{H}\otimes H^{\ast}))\otimes(b_{H}(K)\circ c_{H^{\ast},H})\otimes H)\circ(H\otimes c_{H^{\ast},H^{\ast}}\otimes c_{H,H})\circ(H\otimes H^{\ast}\otimes c_{H,H^{\ast}}\otimes H)\\&\circ(((c_{H^{\ast},H}\otimes H)\circ(H^{\ast}\otimes c_{H,H})\circ(\alpha_{H}\otimes H))\otimes \beta_{H})\circ(H\otimes\delta_{H})\circ\delta_{H}\;\footnotesize\textnormal{(by naturality of $c$)}\\=&((b_{H}(K)\circ(\lambda_{H}\otimes H^{\ast}))\otimes H)\circ(H\otimes((H^{\ast}\otimes b_{H}(K)\otimes H)\circ(\beta_{H}\otimes\alpha_{H})))\circ(H\otimes c_{H,H})\\&\circ(c_{H,H}\otimes H)\circ(H\otimes\delta_{H})\circ\delta_{H}\;\footnotesize\textnormal{(by naturality of $c$)}\\=&((b_{H}(K)\circ(\lambda_{H}\otimes H^{\ast}))\otimes H)\circ(H\otimes((H^{\ast}\otimes b_{H}(K)\otimes H)\circ(\beta_{H}\otimes\alpha_{H})\circ\delta_{H}))\circ\delta_{H}\;\footnotesize\textnormal{(by naturality of $c$ and}\\&\footnotesize\textnormal{cocommutativity and coassociativity of $\delta_{H}$)}\\=&((b_{H}(K)\circ(\lambda_{H}\otimes H^{\ast}))\otimes H)\circ(H\otimes(\varepsilon_{H}\otimes a_{H}(K)))\circ\delta_{H}\;\footnotesize\textnormal{(by (iv) of Definition \ref{PHopf algebra})}\\=&\lambda_{H}\;\footnotesize\textnormal{(by counit property and the adjunction properties)}.
	\end{align*}
From the previous identity we obtain the following:
	\begin{align*}
		&id_{H}\hatast \widehat{\lambda}_{H}\\=& \widehat{\mu}_{H}\circ(H\otimes \widehat{\lambda}_{H})\circ\delta_{H}\;\footnotesize\textnormal{(by definition of $\hatast$)}\\=& \mu_{H}\circ(H\otimes m_{H} )\circ(\delta_{H}\otimes\widehat{\lambda}_{H})\circ\delta_{H}\;\footnotesize\textnormal{(by definition of $\widehat{\mu}_{H}$)}\\=&  \mu_{H}\circ(H\otimes( m_{H} \circ(H\otimes\widehat{\lambda}_{H})\circ\delta_{H}))\circ\delta_{H}\;\footnotesize\textnormal{(by coassociativity of $\delta_{H}$)}\\=& id_{H}\ast\lambda_{H}\;\footnotesize\textnormal{(by \eqref{prop lambdatr 1})}\\=& \varepsilon_{H}\otimes\eta_{H}\;\footnotesize\textnormal{(by \eqref{antipode})}.\qedhere
	\end{align*}
\end{proof}

The aim of the following results will be to prove that the convolution in the opposite direction is also the identity element, i.e., $\widehat{\lambda}_{H}\hatast id_{H}=\varepsilon_{H}\otimes \eta_{H}$.
\begin{lemma}\label{tildealpha de coalgebras}
	If $(H,m_{H})$ is a cocommutative post-Hopf algebra in $\sf{C}$, then $$\widetilde{\alpha}_{H}\coloneqq ( b_{H}(K)\otimes H)\circ(H\otimes\alpha_{H})\colon H\otimes H\rightarrow H$$ is a coalgebra morphism.
\end{lemma}
\begin{proof}
	From (i.1) of Definition \ref{PHopf algebra}, \eqref{ccb} and the naturality of $c$ we can deduce that 
	\begin{equation}\label{mHdeltaHc}
	\delta_{H}\circ m_{H} \circ c_{H,H}=(( m_{H} \circ c_{H,H})\otimes( m_{H} \circ c_{H,H}))\circ (H\otimes c_{H,H}\otimes H)\circ(\delta_{H}\otimes\delta_{H})
	\end{equation}
	holds. Therefore, we obtain that:
	\begin{align*}
	&\delta_{H}\circ\widetilde{\alpha}_{H}\\=& ( b_{H}(K)\otimes\delta_{H})\circ(H\otimes\alpha_{H})\;\footnotesize\textnormal{(by definition of $\widetilde{\alpha}_{H}$)}\\=& \delta_{H}\circ m_{H}\circ c_{H,H}\;\footnotesize\textnormal{(by \eqref{ccb} and \eqref{eq1})}\\=& ((m_{H}\circ c_{H,H})\otimes(m_{H}\circ c_{H,H}))\circ (H\otimes c_{H,H}\otimes H)\circ(\delta_{H}\otimes\delta_{H})\;\footnotesize\textnormal{(by \eqref{mHdeltaHc})}\\=& (\widetilde{\alpha}_{H}\otimes\widetilde{\alpha}_{H})\circ(H\otimes c_{H,H}\otimes H)\circ(\delta_{H}\otimes\delta_{H})\;\footnotesize\textnormal{(by \eqref{ccb}, \eqref{eq1} and definition of $\widetilde{\alpha}_{H}$)}.
	\end{align*}
Moreover, by \eqref{eq1} and (i.2) of Definition \ref{PHopf algebra},  it is easy to prove that $\varepsilon_{H}\circ\widetilde{\alpha}_{H}=\varepsilon_{H}\otimes\varepsilon_{H}$.
\end{proof}
	
Let $(H,m_{H})$ be a post-Hopf algebra in $\sf{C}$ and consider now the morphism $$\widetilde{\beta}_{H}\coloneqq ( b_{H}(K)\otimes H)\circ(H\otimes\beta_{H})\colon H\otimes H\rightarrow H.$$
\begin{lemma}
	Let $(H,m_{H})$ be a post-Hopf algebra in $\sf{C}$. It is satisfied that
	\begin{equation}\label{property for betatilde1}
		\varepsilon_{H}\circ\widetilde{\beta}_{H}=\varepsilon_{H}\otimes\varepsilon_{H}.
	\end{equation}
\end{lemma}
\begin{proof}
	\begin{align*}
		&\varepsilon_{H}\circ\widetilde{\beta}_{H}\\=& ((\varepsilon_{H}\circ\widetilde{\beta}_{H})\otimes\varepsilon_{H})\circ(H\otimes\delta_{H})\;\footnotesize\textnormal{(by counit properties)}\\=& ( b_{H}(K)\otimes\varepsilon_{H}\otimes\varepsilon_{H})\circ(H\otimes\beta_{H}\otimes H)\circ(H\otimes\delta_{H})\;\footnotesize\textnormal{(by definition of $\widetilde{\beta}_{H}$)}\\=& ( b_{H}(K)\otimes(\varepsilon_{H}\circ\widetilde{\alpha}_{H}))\circ(H\otimes\beta_{H}\otimes H)\circ(H\otimes\delta_{H})\;\footnotesize\textnormal{(by the condition of coalgebra morphism for $\widetilde{\alpha}_{H}$)}\\=& ( b_{H}(K)\otimes\varepsilon_{H})\circ(H\otimes((H^{\ast}\otimes  b_{H}(K)\otimes H)\circ(\beta_{H}\otimes\alpha_{H})\circ\delta_{H}))\;\footnotesize\textnormal{(by definition of $\widetilde{\alpha}_{H}$)}\\=& (( b_{H}(K)\otimes\varepsilon_{H})\circ(H\otimes  a_{H}(K)))\otimes\varepsilon_{H}\;\footnotesize\textnormal{(by (iv) of Definition \ref{PHopf algebra})}\\=& \varepsilon_{H}\otimes\varepsilon_{H}\;\footnotesize\textnormal{(by the adjunction properties)}.\qedhere
	\end{align*}
\end{proof}
Let $(H,m_{H})$ be a post-Hopf algebra and suppose that $\widetilde{\beta}_{H}$ satisfies that
\begin{equation}\label{property for betatilde2}
	\delta_{H}\circ\widetilde{\beta}_{H}=(\widetilde{\beta}_{H}\otimes\widetilde{\beta}_{H})\circ(H\otimes c_{H,H}\otimes H)\circ(\delta_{H}\otimes\delta_{H}).
\end{equation}
Note that if \eqref{property for betatilde2} holds, then $\widetilde{\beta}_{H}$ is a coalgebra morphism by the previous lemma.
\begin{lemma}\label{Lema prop widetildelambda}
	Let $(H,m_{H})$ be a cocommutative post-Hopf algebra in $\sf{C}$, then
		\begin{equation}\label{prop lambdatr3}
		\varepsilon_{H}\circ\widehat{\lambda}_{H}=\varepsilon_{H}.
	\end{equation}
	 In addition, if the identity \eqref{property for betatilde2} holds, then 	
	\begin{equation}\label{prop lambdatr2}
		\delta_{H}\circ\widehat{\lambda}_{H}=(\widehat{\lambda}_{H}\otimes\widehat{\lambda}_{H})\circ\delta_{H},
	\end{equation}
	i.e. $\widehat{\lambda}_{H}$ is a coalgebra morphism,
	\begin{equation}\label{prop lambdatr4}
		\widehat{\lambda}_{H}\circ\widehat{\lambda}_{H}=id_{H},
	\end{equation}
	and
	\begin{equation}\label{other conv}\widehat{\lambda}_{H}\hatast id_{H}=\varepsilon_{H}\otimes\eta_{H}.\end{equation}
\end{lemma}
\begin{proof}
	First of all, note that it is straightforward to prove that $\varepsilon_{H}\circ \widehat{\lambda}_{H}=\varepsilon_{H}$ using \eqref{lambdatr cocom}, \eqref{property for betatilde1}, \eqref{u-antip} and counit property. Moreover
\begin{align*}
	&\delta_{H}\circ\widehat{\lambda}_{H}\\=& \delta_{H}\circ\widetilde{\beta}_{H}\circ(\lambda_{H}\otimes H)\circ\delta_{H}\;\footnotesize\textnormal{(by \eqref{lambdatr cocom} and definition of $\widetilde{\beta}_{H}$)}
	\\=& (\widetilde{\beta}_{H}\otimes\widetilde{\beta}_{H})\circ(H\otimes c_{H,H}\otimes H)\circ(\delta_{H}\otimes\delta_{H})\circ (\lambda_{H}\otimes H)\circ\delta_{H}\;\footnotesize\textnormal{(by \eqref{property for betatilde2})}\\=&(\widetilde{\beta}_{H}\otimes\widetilde{\beta}_{H})\circ(H\otimes c_{H,H}\otimes H)\circ(((\lambda_{H}\otimes\lambda_{H})\circ\delta_{H})\otimes\delta_{H})\circ\delta_{H}\;\footnotesize\textnormal{(by \eqref{a-antip} and cocommutativity of $\delta_{H}$)}\\=&((\widetilde{\beta}_{H}\circ(\lambda_{H}\otimes H))\otimes(\widetilde{\beta}_{H}\circ(\lambda_{H}\otimes H)))\circ(H\otimes(c_{H,H}\circ\delta_{H})\otimes H)\circ(H\otimes\delta_{H})\circ\delta_{H}\;\footnotesize\textnormal{(by naturality of $c$ and}\\&\footnotesize\textnormal{coassociativity of $\delta_{H}$)}\\=&((\widetilde{\beta}_{H}\circ(\lambda_{H}\otimes H)\circ\delta_{H})\otimes(\widetilde{\beta}_{H}\circ(\lambda_{H}\otimes H)\circ\delta_{H}))\circ\delta_{H}\;\footnotesize\textnormal{(by cocommutativity and coassociativity of $\delta_{H}$)}\\=&(\widehat{\lambda}_{H}\otimes\widehat{\lambda}_{H})\circ\delta_{H}\;\footnotesize\textnormal{(by \eqref{lambdatr cocom} and definition of $\widetilde{\beta}_{H}$)}.
\end{align*}

As a consequence, we can prove that $\widehat{\lambda}_{H}\circ\widehat{\lambda}_{H}=id_{H}$. Indeed,
	\begin{align*}
		&\widehat{\lambda}_{H}\circ\widehat{\lambda}_{H}\\=& \widehat{\mu}_{H}\circ((\eta_{H}\circ\varepsilon_{H})\otimes(\widehat{\lambda}_{H}\circ\widehat{\lambda}_{H}))\circ\delta_{H}\;\footnotesize\textnormal{(by (co)unit properties)}\\=& \widehat{\mu}_{H}\circ((id_{H}\hatast\widehat{\lambda}_{H})\otimes(\widehat{\lambda}_{H}\circ\widehat{\lambda}_{H}))\circ\delta_{H}\;\footnotesize\textnormal{(by \eqref{convolution1})}\\=& \widehat{\mu}_{H}\circ(H\otimes (\widehat{\mu}_{H}\circ(\widehat{\lambda}_{H}\otimes(\widehat{\lambda}_{H}\circ\widehat{\lambda}_{H}))\circ\delta_{H}))\circ\delta_{H}\;\footnotesize\textnormal{(by coassociativity of $\delta_{H}$ and associativity of $\widehat{\mu}_{H}$)}\\=& \widehat{\mu}_{H}\circ(H\otimes((id_{H}\hatast\widehat{\lambda}_{H})\circ\widehat{\lambda}_{H}))\circ\delta_{H}\;\footnotesize\textnormal{(by \eqref{prop lambdatr2})}\\=& \widehat{\mu}_{H}\circ(H\otimes(\eta_{H}\circ\varepsilon_{H}\circ\widehat{\lambda}_{H}))\circ\delta_{H}\;\footnotesize\textnormal{(by \eqref{convolution1})}\\=& \widehat{\mu}_{H}\circ(H\otimes(\eta_{H}\circ\varepsilon_{H}))\circ\delta_{H}\;\footnotesize\textnormal{(by \eqref{prop lambdatr3})}\\=& id_{H}\;\footnotesize\textnormal{(by (co)unit properties)}.
	\end{align*}
	
To finish, we will see that $\widehat{\lambda}_{H}\hatast id_{H}=\varepsilon_{H}\otimes\eta_{H}$. Indeed,
	\begin{align*}
		&\widehat{\lambda}_{H}\hatast id_{H}\\=& \widehat{\mu}_{H}\circ(\widehat{\lambda}_{H}\otimes H)\circ\delta_{H}\;\footnotesize\textnormal{(by definition of $\hatast$)}\\=&  \widehat{\mu}_{H}\circ(\widehat{\lambda}_{H}\otimes (\widehat{\lambda}_{H}\circ\widehat{\lambda}_{H}))\circ\delta_{H}\;\footnotesize\textnormal{(by \eqref{prop lambdatr4})}\\=& (id_{H}\hatast\widehat{\lambda}_{H})\circ\widehat{\lambda}_{H}\;\footnotesize\textnormal{(by \eqref{prop lambdatr2})}\\=& \eta_{H}\circ\varepsilon_{H}\circ\widehat{\lambda}_{H}\;\footnotesize\textnormal{(by \eqref{convolution1})}\\=& \varepsilon_{H}\otimes\eta_{H}\;\footnotesize\textnormal{(by \eqref{prop lambdatr3})}.\qedhere
	\end{align*}
\end{proof}
\begin{theorem}\label{HHatHA}
Let $(H,m_{H})$ be a cocommutative post-Hopf algebra in $\sf{C}$. If the identity \eqref{property for betatilde2} holds, then $\widehat{H}=(H,\eta_{H},\widehat{\mu}_{H},\varepsilon_{H},\delta_{H},\widehat{\lambda}_{H})$ is a cocommutative Hopf algebra in $\sf{C}$. This particular Hopf algebra structure is called the subadjacent Hopf algebra of $(H,m_{H})$.
\end{theorem}
\begin{proof}
It is a direct consequence of Theorem \ref{Th bialg subadj} and equalities \eqref{convolution1} and \eqref{other conv}.
\end{proof}
So, we have deduced that it is possible to obtain from any cocommutative post-Hopf algebra satisfying \eqref{property for betatilde2} another Hopf algebra structure whose underlying coalgebra is the same as that of $H$. Therefore, at this point it is natural to wonder whether $\widehat{\mathbb{H}}=(H,\widehat{H})$ is a Hopf brace in $\sf{C}$. The following theorem solves this question.
\begin{theorem}\label{Hhat brace}
	Let $(H,m_{H})$ be a cocommutative post-Hopf algebra in $\sf{C}$. If the identity \eqref{property for betatilde2} is satisfied, then
	\[\widehat{\mathbb{H}}=(H,\widehat{H})\]
	is a cocommutative Hopf brace in $\sf{C}$.
\end{theorem}
\begin{proof} By Theorem \ref{HHatHA}, to prove that $\widehat{\mathbb{H}}=(H,\widehat{H})$ is a Hopf brace we only have to show that (iii) of Definition \ref{H-brace} holds. Note that
\begin{equation}\label{Ghat=mH}
\widehat{\Gamma}_{H}=m_{H}.
\end{equation}
Indeed,
	\begin{align*}
		&\widehat{\Gamma}_{H}\\=&\mu_{H}\circ(\lambda_{H}\otimes\widehat{\mu}_{H})\circ(\delta_{H}\otimes H)\;\footnotesize\textnormal{(by definition of $\widehat{\Gamma}_{H}$)}\\=& \mu_{H}\circ(\lambda_{H}\otimes(\mu_{H}\circ(H\otimes m_{H} )\circ(\delta_{H}\otimes H)))\circ(\delta_{H}\otimes H)\;\footnotesize\textnormal{(by definition of $\widehat{\mu}_{H}$)}\\=& \mu_{H}\circ((\lambda_{H}\ast id_{H})\otimes m_{H} )\circ(\delta_{H}\otimes H)\;\footnotesize\textnormal{(by coassociativity of $\delta_{H}$ and associativity of $\mu_{H}$)}\\=& \mu_{H}\circ((\eta_{H}\circ\varepsilon_{H})\otimes m_{H} )\circ(\delta_{H}\otimes H)\;\footnotesize\textnormal{(by \eqref{antipode})}\\=&  m_{H} \;\footnotesize\textnormal{(by (co)unit property)}.
	\end{align*}
	So, we obtain the following:
	\begin{align*}
		&\mu_{H}\circ(\widehat{\mu}_{H}\otimes\widehat{\Gamma}_{H})\circ(H\otimes c_{H,H}\otimes H)\circ(\delta_{H}\otimes H\otimes H)\\=& \mu_{H}\circ((\mu_{H}\circ(H\otimes m_{H} )\circ(\delta_{H}\otimes H))\otimes m_{H} )\circ(H\otimes c_{H,H}\otimes H)\circ(\delta_{H}\otimes H\otimes H)\\&\footnotesize\textnormal{(by definition of $\widehat{\mu}_{H}$ and \eqref{Ghat=mH})}\\=& \mu_{H}\circ(H\otimes(\mu_{H}\circ( m_{H} \otimes m_{H} )\circ(H\otimes c_{H,H}\otimes H)\circ(\delta_{H}\otimes H\otimes H)))\circ(\delta_{H}\otimes H\otimes H)\\&\footnotesize\textnormal{(by associativity of $\mu_{H}$ and coassociativity of $\delta_{H}$)}\\=& \mu_{H}\circ(H\otimes m_{H} )\circ(\delta_{H}\otimes \mu_{H})\;\footnotesize\textnormal{(by (iii) of Definition \ref{PHopf algebra})}\\=& \widehat{\mu}_{H}\circ(H\otimes \mu_{H})\;\footnotesize\textnormal{(by definition of $\widehat{\mu}_{H}$)}.\qedhere
	\end{align*}
\end{proof}
It is possible to interpret the previous result in the following sense: If $\sf{cocPost}\textnormal{-}\sf{Hopf}^{\star}$ denotes the full subcategory of cocommutative post-Hopf algebras such that \eqref{property for betatilde2} holds, and $\sf{cocHBr^{f}}$ denotes the category of finite cocommutative Hopf braces, then a functor
$Q\colon \sf{cocPost}\textnormal{-}\sf{Hopf}^{\star}\longrightarrow \sf{cocHBr^{f}}$ exists which
acts on objects by $Q((H,m_{H}))=\widehat{\mathbb{H}}$ and on morphisms by the identity. To see that $Q$ is well-defined on morphisms, we have to prove that if $f\colon (H,m_{H})\rightarrow (B,m_{B})$ is a morphism in $\sf{cocPost}\textnormal{-}\sf{Hopf}$, then $f$ is a morphism of Hopf braces between $\widehat{\mathbb{H}}$ and $\widehat{\mathbb{B}}$. Indeed:
\begin{align*}
&f\circ\widehat{\mu}_{H}\\=& f\circ\mu_{H}\circ(H\otimes m_{H})\circ(\delta_{H}\otimes H)\;\footnotesize\textnormal{(by definition of $\widehat{\mu}_{H}$)}\\=& \mu_{B}\circ(f\otimes f)\circ(H\otimes m_{H})\circ(\delta_{H}\otimes H)\;\footnotesize\textnormal{(by the condition of algebra morphism for $f\colon H\rightarrow B$)}\\=& \mu_{B}\circ(B\otimes m_{B})\circ(((f\otimes f)\circ\delta_{H})\otimes f)\;\footnotesize\textnormal{(by \eqref{CondMorPHopf})}\\=& \mu_{B}\circ(B\otimes m_{B})\circ(\delta_{B}\otimes B)\circ (f\otimes f)\;\footnotesize\textnormal{(by the condition of coalgebra morphism for $f$)}\\=& \widehat{\mu}_{B}\circ(f\otimes f)\;\footnotesize\textnormal{(by definition of $\widehat{\mu}_{B}$)}.
\end{align*}
The following theorem is the main result of this section. 
\begin{theorem}\label{iso2}
The categories $\sf{cocPost}\textnormal{-}\sf{Hopf}^{\star}$ and $\sf{cocBT^{f}}$ are isomorphic.
\end{theorem}
\begin{proof}
At first, let's consider the following commutative diagram of functors:
\[\xymatrix{&\sf{cocPost}\textnormal{-}\sf{Hopf}^{\star}\ar[rr]^-{Q} & &\sf{cocHBr^{f}}\ar@<1ex>[dd]^-{G''}\\
& & &\rotatebox{90}{$\simeq$}\\
& & &\sf{cocBT^{f}},\ar[lluu]_-{P'}\ar@<1ex>[uu]^-{F''}
}\]
where $F''$ and $G''$ are the restrictions of functors $F'$ and $G'$ introduced in Corollary \ref{Cor iso cocHB cocTB} to the subcategories of finite objects, and $P'$ is the restriction of functor $P$ to $\sf{cocBT^{f}}$. 

To begin with, we are going to see that $P'$ is well-defined on objects. That is to say, we have to prove that if $(H,\gamma_{H},T_{H})$ is a finite cocommutative brace triple, then the post-Hopf algebra $(H,\gamma_{H})$ satisfies \eqref{property for betatilde2}. In this situation, by Theorem \ref{inversebetaBT} and \eqref{TH2=ID}, \begin{equation}\label{betaparticular}\beta_{H}=\alpha_{H}\circ T_{H},\end{equation}and then, \begin{equation}\label{betatildeparticular}\widetilde{\beta}_{H}=\widetilde{\alpha}_{H}\circ (H\otimes T_{H}).\end{equation} Therefore, \eqref{property for betatilde2} follows by:
\begin{align*}
&\delta_{H}\circ\widetilde{\beta}_{H}\\=&\delta_{H}\circ\widetilde{\alpha}_{H}\circ(H\otimes T_{H})\;\footnotesize\textnormal{(by $\widetilde{\beta}_{H}=\widetilde{\alpha}_{H}\circ (H\otimes T_{H})$)}\\=& (\widetilde{\alpha}_{H}\otimes \widetilde{\alpha}_{H})\circ(H\otimes c_{H,H}\otimes H)\circ(\delta_{H}\otimes(\delta_{H}\circ T_{H}))\;\footnotesize\textnormal{(by Lemma \ref{tildealpha de coalgebras})}\\=& (\widetilde{\alpha}_{H}\otimes \widetilde{\alpha}_{H})\circ(H\otimes c_{H,H}\otimes H)\circ(\delta_{H}\otimes((T_{H}\otimes T_{H})\circ\delta_{H}))\;\footnotesize\textnormal{(by (vi.1) of Definition \ref{BTdef}}\\&\footnotesize\textnormal{and cocommutativity of $\delta_{H}$)}\\=& (\widetilde{\beta}_{H}\otimes \widetilde{\beta}_{H})\circ(H\otimes c_{H,H}\otimes H)\circ(\delta_{H}\otimes\delta_{H})\;\footnotesize\textnormal{(by naturality of $c$ and \eqref{betatildeparticular})}.\end{align*}

Taking into account functors $P'$, $Q$ and $G''$, on the one hand, we have that
\begin{align*}
&(P'\circ(G'' \circ Q))((H,m_{H}))\\=&(P'\circ G'')(\widehat{\mathbb{H}})\;\footnotesize\textnormal{(by definition of $Q$)}\\=&P'((H,m_{H},\widehat{\lambda}_{H}))\;\footnotesize\textnormal{(by definition of $G''$ and \eqref{Ghat=mH})}\\=&(H,m_{H})\;\footnotesize\textnormal{(by definition of $P'$)}.
\end{align*}
So, $P'\circ(G''\circ Q)=\sf{id}_{\sf{cocPost}\textnormal{-}\sf{Hopf}^{\star}}$. On the other hand,
\begin{align*}
&((G''\circ Q)\circ P')((H,\gamma_{H},T_{H}))\\=&(G''\circ Q)((H,\gamma_{H}))\;\footnotesize\textnormal{(by definition of $P'$)}\\=&G''(\widehat{\mathbb{H}})\;\footnotesize\textnormal{(by definition of $Q$)}\\=&(H,\gamma_{H},\widehat{\lambda}_{H})\;\footnotesize\textnormal{(by definition of $G''$ and \eqref{Ghat=mH})}\\=&(H,\gamma_{H},T_{H}),
\end{align*}
where the last equality is due to the fact that $\widehat{\lambda}_{H}=T_{H}$. Indeed, firstly note that
\begin{align*}
&\widehat{\lambda}_{H}\\=&( b_{H}(K)\otimes H)\circ(\lambda_{H}\otimes\beta_{H})\circ\delta_{H}\;\footnotesize\textnormal{(by \eqref{lambdatr cocom})}\\=&( b_{H}(K)\otimes H)\circ(\lambda_{H}\otimes(\alpha_{H}\circ T_{H}))\circ\delta_{H}\;\footnotesize\textnormal{(by Theorem \ref{inversebetaBT} and \eqref{TH2=ID})}\\=&\gamma_{H}\circ c_{H,H}\circ(\lambda_{H}\otimes T_{H})\circ\delta_{H}\;\footnotesize\textnormal{(by \eqref{eq1} and \eqref{ccb})}\\=&\gamma_{H}\circ(T_{H}\otimes\lambda_{H})\circ\delta_{H}\;\footnotesize\textnormal{(by naturality of $c$ and cocommutativity of $\delta_{H}$)}.
\end{align*}
As a result,
\begin{align*}
&id_{H}\ast_{\scalebox{0.6}{\sf{BT}}} \widehat{\lambda}_{H}\\=&\mu_{H}^{\scalebox{0.6}{\sf{BT}}}\circ(H\otimes\widehat{\lambda}_{H})\circ\delta_{H}\;\footnotesize\textnormal{(by definition of $\ast_{\scalebox{0.6}{\sf{BT}}}$)}\\=&\mu_{H}\circ(H\otimes\gamma_{H})\circ(\delta_{H}\otimes(\gamma_{H}\circ(T_{H}\otimes\lambda_{H})\circ\delta_{H}))\circ\delta_{H}\;\footnotesize\textnormal{(by definition of $\mu_{H}^{\scalebox{0.6}{\sf{BT}}}$ and previous equality)}\\=&\mu_{H}\circ(H\otimes(\gamma_{H}\circ((\mu_{H}\circ(H\otimes \gamma_{H})\circ(\delta_{H}\otimes H))\otimes H)))\circ(\delta_{H}\otimes((T_{H}\otimes\lambda_{H})\circ\delta_{H}))\circ\delta_{H}\\&\footnotesize\textnormal{(by (iv) of Definition \ref{BTdef})}\\=&\mu_{H}\circ(H\otimes(\gamma_{H}\circ(\mu_{H}\otimes\lambda_{H})\circ(H\otimes(\gamma_{H}\circ(H\otimes T_{H})\circ\delta_{H})\otimes H)))\circ(\delta_{H}\otimes\delta_{H})\circ\delta_{H}\\&\footnotesize\textnormal{(by coassociativity of $\delta_{H}$)}\\=&\mu_{H}\circ(H\otimes(\gamma_{H}\circ((id_{H}\ast\lambda_{H})\otimes\lambda_{H})\circ\delta_{H}))\circ\delta_{H}\;\footnotesize\textnormal{(by (vi.4) of Definition \ref{BTdef} and coassociativity of $\delta_{H}$)}\\=&id_{H}\ast\lambda_{H}\;\footnotesize\textnormal{(by \eqref{antipode}, counit property and (v) of Definition \ref{BTdef})}\\=&\varepsilon_{H}\otimes\eta_{H}\;\footnotesize\textnormal{(by \eqref{antipode})},
\end{align*}
which implies, due to the uniqueness of the antipode for the Hopf algebra $H_{\sf{BT}}$, that $\widehat{\lambda}_{H}=T_{H}.$ Hence, $(G''\circ Q)\circ P'=\sf{id}_{\sf{cocBT^{f}}}$.
\end{proof}
\begin{corollary}
Categories $\sf{cocBT^{f}}$, $\sf{cocHBr^{f}}$ and $\sf{cocPost}\textnormal{-}\sf{Hopf}^{\star}$ are isomorphic.
\end{corollary}
\begin{proof}
It is a direct consequence of the previous theorem and Corollary \ref{Cor iso cocHB cocTB}.
\end{proof}
\section*{Funding}
The  authors were supported by  Ministerio de Ciencia e Innovaci\'on of Spain. Agencia Estatal de Investigaci\'on. Uni\'on Europea - Fondo Europeo de Desarrollo Regional (FEDER). Grant PID2020-115155GB-I00: Homolog\'{\i}a, homotop\'{\i}a e invariantes categ\'oricos en grupos y \'algebras no asociativas.

Moreover, José Manuel Fernández Vilaboa and Brais Ramos Pérez were funded by Xunta de Galicia, grant ED431C 2023/31 (European FEDER support included, UE).

Also, Brais Ramos Pérez was financially supported by Xunta de Galicia Scholarship ED481A-2023-023.

\bibliographystyle{amsalpha}

\end{document}